\documentclass[leqno,a4paper]{amsart}
\usepackage{amsfonts,amssymb,graphicx,hyperref}
\usepackage{soul}
\usepackage{xcolor}
\hypersetup{colorlinks=true,linkcolor=red!50!black}
\hypersetup{anchorcolor=green,citecolor=blue!80!black,urlcolor=red!50!black,filecolor=magenta,pdftoolbar=true}
\usepackage[russian,english]{babel}

\allowdisplaybreaks


\usepackage{breqn}
\setkeys{breqn}{compact}

\makeatletter
\def\big{\bBigg@{1.25}}
\makeatother
\usepackage[utf8]{inputenc}
\usepackage[T1]{fontenc}
\usepackage[nobreak]{cite}
\usepackage{mathtools,enumitem}
\usepackage[capitalize,nameinlink]{cleveref}
\Crefname{enumi}{}{}
\Crefname{subsection}{Subsection}{Subsections}

\newtheorem{theorem}{Theorem}[section]
\newtheorem{lemma}[theorem]{Lemma}

\newtheorem{corollary}[theorem]{Corollary}
\newtheorem{problem}[theorem]{Problem}
\theoremstyle{definition}

\newtheorem{example}[theorem]{Example}
\theoremstyle{remark}
\newtheorem{remark}[theorem]{Remark}

\crefname{equation}{}{}
\numberwithin{equation}{section}

\DeclarePairedDelimiter{\absof}{\lvert}{\rvert} 
\DeclarePairedDelimiter{\paren}{\lparen}{\rparen} 
\DeclarePairedDelimiter{\setof}{\lbrace}{\rbrace}

\DeclareMathOperator{\set}{set}

\newcommand{\abs}[1]{\left|#1\right|}

\newcommand{\Z}{\mathbb Z}

\DeclareMathOperator{\stab}{\mathbf S}

\newcommand{\orb}{\mathbf O}

\DeclareMathOperator{\ext}{ext} 
\DeclareMathOperator{\chu}{c-hull} 

\newcommand{\chS}{\chu_S} 
\newcommand{\chR}{\chu_R}

\usepackage{lipsum}
\makeatletter
\newcommand{\authorfootnotes}{\renewcommand\thefootnote{\@fnsymbol\c@footnote}}
\g@addto@macro{\endabstract}{\@setabstract}
\makeatother

\begin{document}
	
\begin{center}
\Large
{\bf\MakeUppercase{Betweenness isomorphism classes of circles with finitely many points inside}}\par \bigskip\bigskip

\normalsize
\authorfootnotes

Martin Dole\v{z}al\footnote{Research supported by the GA\v{C}R projekt EXPRO 20-31529X and by the Czech Academy of Sciences (RVO 67985840).}\textsuperscript{1}, 
Jan Kol\'{a}\v{r}\footnote{Research supported by the Czech Academy of Sciences (RVO 67985840).}\textsuperscript{1}, and
Janusz Morawiec\footnote{Research supported by the University of Silesia Mathematics Department (Iterative Functional Equations and Real Analysis program).}\textsuperscript{2} 
\par \bigskip

\textsuperscript{1}Institute of Mathematics, Czech Academy of Sciences, \v{Z}itn\'{a} 25, 115 67 Praha 1, Czech Republic,\par email addresses: dolezal@math.cas.cz,\, kolar@math.cas.cz \par 
\textsuperscript{2}Institute of Mathematics, University of Silesia, Bankowa 14, PL-40-007 Katowice, Poland,\par email address: janusz.morawiec@us.edu.pl \par \bigskip	

\today
\end{center}
\medskip	
	 
\begin{abstract}
We give a necessary and sufficient condition for two circles, each with finitely many points added inside, to be betweenness isomorphic. 
We fully characterize the betweenness isomorphism classes in the family consisting of all circles with three collinear points inside.
\end{abstract}
\bigskip 

\bigskip\bigskip	
	
\tableofcontents

\section{Introduction}
Many interesting, natural and important mathematical concepts can be defined in the abstract setting. One of them is a ternary relation called \emph{betweenness}, which is widely studied in connection with broadly understood geometry (see, e.g., \cite{Soltan1984, AdelekeNeumann1998, Pambuccian2011} and the references therein). 
The study of this topic, including axiomatization, goes back to \cite{Pasch1882, HuntingtonKline1917, Huntington1924}, where it was introduced for the purposes of plane geometry. In the next years, betweenness was considered not only as a purely theoretical concept of the plane geometry, but it was also studied in different contexts with correlations to mathematical objects such as algebra (see, e.g., \cite{Hedlikova1983, ChajdaKolarikLanger2013, JostWenzel2023}), convex structures (see, e.g., \cite{Vel1993, Kubis2002, Chvatal2009, AndersonBankstonMcCluskey2021}), graphs (see, e.g., \cite{MorganaMulder2002,ChangatNarasimha-ShenoiSeethakuttyamma2019, Courcelle2020, Courcelle2021}), lattices (see, e.g., \cite{SmileyTransue1943, DuvelmeyerWenzel2004}), metric and normed spaces (see, e.g., \cite{Toranzos1971, DiminnieWhite1981, Simovici2009, BankstonMcCluskey2023}), ordered sets (see, e.g., \cite{Sholander1952, Fishburn1971, ZhangPerez-FernandezBeats2019,Lihova2000}),  topological structures (see, e.g., \cite{Bankston2013, Bankston2015, Shakir2023}), and many others (see, e.g., \cite{ChvatalWu2012, BrunoMcCluskeySzeptycki2017, Shi2022, ZhaoZhaoAhang2023}).

\emph{Linear betweenness}
(induced by a fixed linear ordering on a given line)
has three essential generalizations: \emph{algebraic betweenness} (applied to vector 
spaces), \emph{metric betweenness} (considered on semimetric spaces), and 
\emph{lattice betweenness} (studied on lattice objects). A comprehensive comparison 
of these three relations can be found in \cite{Smiley1943}. In this paper we are only
interested in algebraic betweenness on certain subsets of the Euclidean plane 
(called also \emph{Euclidean betweenness}) defined as follows: Given a subset $A$ of 
Euclidean space, we say that a point $x\in A$ is \emph{between} points $a\in A$ and 
$b\in A$ (denoted by $B_A(a,x,b)$ in the language of the betweenness relation) if and only if $x=(1-\lambda)a+\lambda b$ for some $\lambda\in[0,1]$. Denoting by $[a,b]$ 
the (closed) linear segment connecting $a$ and $b$, i.e., 
$[a,b]=\{\lambda a+(1-\lambda)b:\lambda\in[0,1]\}$, we can equivalently say that 
a point $x\in A$ is between $a\in A$ and $b\in A$ if and only if $x\in[a,b]$. 
Euclidean betweenness can be easily generalized to vector spaces over arbitrary 
ordered fields; however, we are not going to explore this direction.

Given two betweenness structures $(X,B_X)$ and $(Y,B_Y)$, a map $f\colon X\to Y$ is said to be:
\begin{enumerate}
\item[-] \emph{betweenness preserving} if 
\begin{equation*}
B_X(a,x,b) \implies B_Y(f(a),f(x),f(b))
\end{equation*}
for all $a,x,b \in X$ (see e.g. \cite{HouMcColm2008});
\item[-] \emph{betweenness isomorphism} if it is a bijection and 
\begin{equation*}
B_X(a,x,b) \iff B_Y(f(a),f(x),f(b))
\end{equation*}	
for all $a,x,b \in X$ (cf. \cite{Hedlikova1981} for the lattice case). 
\end{enumerate}
In the case of Euclidean betweenness, we will also use the term 
\emph{betweenness homomorphism} instead of betweenness preserving 
map. For example, if $S,R\subset\mathbb R^2$, we call a map 
$f\colon S\to R$ betweenness homomorphism if, for all $a,b,c\in S$ 
such that $c\in[a,b]$, it holds $f(c)\in[f(a),f(b)]$. 
Denote by $(a,b)$ the open segment, i.e., 
$(a,b):=[a,b]\setminus\{a,b\}$ and note that for any bijection 
$f\colon S\to R$ the condition $c\in[a,b]\iff f(c)\in[f(a),f(b)]$ 
is equivalent to the condition $c\in(a,b)\iff f(c)\in(f(a),f(b))$. 
Therefore, throughout this paper, a \emph{betweenness isomorphism
of $S$ and $R$} is a bijective map $f\colon S\to R$ such that for 
every $a,b,c\in S$, it holds $c\in(a,b)\iff f(c)\in(f(a),f(b))$.

We say that sets $S,R\subset\mathbb R^2$ are \emph{betweenness isomorphic} if there exists a betweenness isomorphism of $S$ and $R$.
Note that every betweenness preserving map that is a bijection between 
linearly ordered sets is automatically a betweenness isomorphism, but
this is not the case for Euclidean betweenness; indeed, each bijection $f\colon K\to [0,1)\times\{0\}\subset\mathbb R^2$, where 
$K\subset\mathbb R^2$ is the unit circle centred at the origin, 
is a betweenness homomorphism; however, its inverse $f^{-1}$ 
is very far from being a betweenness homomorphism. 

Among subsets of the plane, one can easily find examples of those 
$A\subset\mathbb R^2$ for which the Euclidean betweenness is 
\emph{discrete}, i.e., $B_A(a,x,b)$ holds if and only if $x\in\{a,b\}$
or equivalently $A$ does not contain three distinct collinear points. 
A typical example of such a set is the unit circle in the plane. 
Another example is any set $M$ first constructed in \cite{Mazurkiewicz1914}) that has exactly two points in common with 
every line in the plane. Note that the set $M$ (described above) and 
$K$ (the unit circle in the plane) are trivially betweenness 
isomorphic. Therefore, finding all subsets of the Euclidean plane that 
are betweenness isomorphic to a given set is rather difficult. But one 
can try to decide when two given subsets of the Euclidean plain of a 
given form (e.g., two sets such that each of them consists exactly of 
$l$ points or lines) are betweenness isomorphic. Moreover, one can ask 
if any classification of a family of sets of the same form (e.g., the 
family of all sets consisting exactly of $l$ points or lines) in the 
language of betweenness isomorphism would be possible. To our best 
knowledge, Wies{\l}aw Kubi\'{s} was one of the first who posed
(in personal communication) two questions in this direction. 
The first one reads as follows.

\begin{problem}\label{prob1x}
Let $l$ be a natural number. Let $S,R$ be subsets of the Euclidean plane,
each of them consisting of a circle and $l$ distinct points in the interior of the circle.
\begin{enumerate}
\item[{\rm (A)}] When are the sets $S$ and $R$ betweenness isomorphic?
\item[{\rm (B)}] How many classes of betweenness isomorphism of sets as above are there, and how to characterize them?
\end{enumerate}  
\end{problem}

The second question is, in spirit, the same as the above one,
with points in the interior replaced by another concentric circle. 

\begin{problem}\label{prob2x}
Let $S,R$ be subsets of the Euclidean plane, each of them consisting of two
distinct concentric circles.
\begin{enumerate}
\item[{\rm (A)}] When are the sets $S$ and $R$ betweenness isomorphic?
\item[{\rm (B)}] How many classes of betweenness isomorphism of sets as above are there, and how to characterize them?
\end{enumerate}  
\end{problem}

Let us note that Problem 3.29.3 on page 66 in \cite{Vel1993} is of a similar type as it concerns the classification of betweenness isomorphism classes of convex polytopes in Euclidean spaces.

A classification of betweenness preserving maps defined on convex planar sets was 
recently obtained in \cite{KubisMorawiecZurcher2022}. Namely, either the image is 
contained in the union of a line, and a single point outside of the line, or the image 
consists of certain five points, or else the mapping is a partial homography, i.e., 
a mapping that can be extended to a homography. Additionally, if the domain is an open
convex set, then either the image is contained in a line or else the mapping is a partial
homography. Unfortunately, this classification cannot be applied to answer \Cref{prob1x} 
because the problem does not concern convex subsets of the Euclidean plane.

The present paper aims to answer question (A) of \Cref{prob1x} and provides some partial
answers to question (B). More precisely, in \Cref{sec3}, we give a necessary and 
sufficient condition for two circles, each with $l$ points in its interiors, to be 
betweenness isomorphic (see \Cref{t:characterization}). We also formulate a useful 
consequence (see \Cref{c:characterization}), which we then apply in \Cref{sec4} to 
answer question (B) of \Cref{prob1x} in the cases $l=1$ (see \Cref{t:l=1}), $l=2$  
(see \Cref{t:l=2}), and $l=3$ assuming additionally that the points in the interior 
of each of the circles are collinear (see \Cref{t:isoclasses}).

\section{Preliminaries}

\subsection{Notation}
Let $S\subset\mathbb R^2$. We say that a point $c\in S$ is \emph{extreme} in $S$ if there
are no $a,b\in S$ such that $c\in(a,b)$. We denote by $\ext(S)$ the set of all extreme 
points in $S$.
 
Let $A\subset S\subset\mathbb R^2$. We say that the set $A$ is \emph{collinearly closed} 
in $S$ if $c\in A$ whenever $c\in S$ and there exist distinct $a,b\in A\setminus\{c\}$ such that $a,b,c$ are collinear. 
We define the \emph{collinear hull} $\chS(A)$ of $A$ in $S$ as the smallest set which 
is collinearly closed in $S$ and which contains $A$.
 
\begin{lemma}\label{l:ext-ch}
Let $S,R\subset\mathbb R^2$ and $A\subset S$. Suppose that $f\colon S\to R$ is a betweenness isomorphism. Then:
\begin{enumerate}[label={\rm (\roman*)}]
\item\label{ext} $f(\ext(S))=\ext(R)$,
\item\label{ch} $f(\chS(A))=\chR(f(A))$.
\end{enumerate} 
\end{lemma}
	
\begin{proof}
\ref{ext} For every $c\in S$, it holds
\begin{equation*}
\begin{split}
c\notin\ext(S)&\iff\exists a,b\in S:c\in(a,b)\iff\exists a,b\in S:f(c)\in(f(a),f(b))\\
&\iff\exists d,e\in R:f(c)\in(d,e)\iff f(c)\notin\ext(R),
\end{split}
\end{equation*}
which proves our assertion.
		
\ref{ch} Fix arbitrary $B\subset S$. Then
\begin{equation*}
\begin{split}
B\text{ is collinearly}&\text{ closed in }S\iff\forall a,b\in B\ \,\forall c\in S:\text{ if }a,b,c\text{ are pairwise distinct collinear, then }c\in B\\
\iff&\forall a,b\in B\ \,\forall c\in S:\text{ if }f(a),f(b),f(c)\text{ are  pairwise distinct collinear, then }f(c)\in f(B)\\
\iff&\forall d,e\in f(B)\ \,\forall g\in R:\text{ if }d,e,g\text{ are pairwise distinct collinear, then }g\in f(B)\\
\iff&f(B)\text{ is collinearly closed in }R.
\end{split}
\end{equation*}
In particular, the set $f(\chS(A))$ is collinearly closed in $R$ and contains $f(A)$.
Hence $f(\chS(A))\supset \chR(f(A))$. As $f^{-1}$ is also a betweenness isomorphism,
the same argument gives $f^{-1}(\chR(f(A)))\supset \chS(f^{-1}(f(A)))$.
As $f^{-1}$ is a bijection, the latter implies $\chR(f(A))\supset f(\chS(A))$.
Therefore, the two sets are equal.
\end{proof}

\subsection{The group $G_l$}\label{subsec:group}
Let us fix $l\in\mathbb N$. Any finite sequence of elements of $\{1,\ldots,l\}$ will be 
called a \emph{configuration}. A configuration $(i_1,\ldots,i_n)$ will be called 
\emph{irreducible} if $i_k\neq i_{k+1}$ for every $k\in\{1,\ldots,n-1\}$; otherwise the 
configuration will be called \emph{reducible}. Note that the empty sequence $\emptyset$ 
(of length $n=0$) is an irreducible configuration. (Also, all configurations of length 
$n=1$ are irreducible.)
	
Let $G_l$ be the set of all irreducible configurations. We consider the binary operation on $G_l$ given by concatenation followed by reduction (if necessary), namely
\begin{equation}
\label{e:Op}
(i_1,\ldots,i_n)(j_1,\ldots,j_m):=
(i_1,\ldots,i_{n-k}, j_{1+k},\ldots,j_m),
\end{equation}
where
\[k=\max\left\{\strut r\in \{0,1,2,\dots, \min(n,m)\} :
i_{n-p} = j_{1+p} \text{ for all } 0\le p<r  \right\}.\]   
Let us explain the meaning of~\eqref{e:Op}.
After concatenating two irreducible configurations $(i_1,\ldots,i_n)$, $(j_1,\ldots,j_m)$ we obtain a new configuration $(i_1,\ldots,i_n,j_1,\ldots,j_m)$.
If this new configuration is irreducible then no reduction is necessary.
Otherwise, it holds $i_n=j_1$. In this case, we `remove' the elements $i_n,j_1$ from the configuration, so that we obtain another configuration $(i_1,\ldots,i_{n-1},j_2,\ldots,j_m)$.
If this configuration is irreducible then the reduction is completed.
Otherwise, it holds $i_{n-1}=j_2$. In that case, we `remove' the elements $i_{n-1},j_2$, so that we obtain yet another configuration $(i_1,\ldots,i_{n-2},j_3,\ldots,j_m)$. We repeat this process until we obtain an irreducible (possibly empty) configuration, which will be the outcome of our reduction.
 
\begin{lemma}
The set $G_l$, together with the binary operation given by concatenation followed by reduction, is a group.
\end{lemma}

\begin{proof}
Associativity is easy to check.
The identity element is the empty configuration.
The inverse of $(i_1,\ldots,i_n)$ is $(i_n,\ldots,i_1)$.
\end{proof}

\subsection{The group action of $G_l$}\label{subsec:reversionETC}
Fix $l\in\mathbb N$. Let $C$ be a circle in the plane and let $\{c_1,\ldots,c_l\}$ be a set of points in the interior of $C$.

Following \cite{Kocik2013}, for every point $X$ in the interior of the circle $C$ we define $P_X\colon C\to C$, called the
\emph{reversion} map (through the point $X$), as follows: if $c\in C$, then $P_X(c)$ is the unique point of $C$ such that $X\in(c,P_X(c))$.
It is clear that for every point $X$ from the interior of the circle $C$ and every point $c\in C$,
we have $(P_X\circ P_X)(c)=c$, i.e., each reversion map is an involution. To shorten the notation, for every $i\in\{1,\ldots,l\}$, we will use the symbol $R_i$ for the reversion map through the point $c_i$, i.e., we put $R_i:=P_{c_i}$.

Now, we define a map $\alpha\colon C\times G_l\to C$ putting 
\[\alpha(c,\emptyset)=c\quad\text{and}\quad\alpha(c,(i_1,\dots,i_n))=(R_{i_n}\circ\dots\circ R_{i_1})(c)\quad\text{for every }(i_1,\dots,i_n)\neq\emptyset.\]

\begin{lemma}\label{l:ra}
The map $\alpha$ is a right action of the group $G_l$ on $C$.
\end{lemma}
	
\begin{proof}
One just needs to check that $\alpha(\alpha(c,g),g')=\alpha(c,gg')$ for all $c\in C$ and $g,g'\in G_l$, which is clear by the definition of $\alpha$. 
\end{proof}
	
In the following, we write shortly $c\cdot g$ instead of $\alpha(c,g)$.
	
For every $c\in C$, let
\[\orb(c)=\left\{c\cdot g:g\in G_l\right\}\]
be the \emph{orbit} of $c$, and let
\[\stab(c)=\left\{g\in G_l:c\cdot g=c\right\}\]
be the \emph{stabilizer} of $c$.

By standard properties of a group action, we have that for all $c,c'\in C$, either $\orb(c)=\orb(c')$ or $\orb(c)\cap \orb(c')=\emptyset$. Thus the set of orbits of points $c\in C$ forms a partition of $C$. In fact, the orbits are the
equivalence classes of the relation $\sim$ on $C$ defined as follows:
\[c\sim c'\iff\exists g\in G_l: c=c'\cdot g.\] 
Note also that the orbit of any point $c\in C$ is a finite or countable set.

Of course, the action $\alpha$ depends on the precise choice of the points $c_1,\ldots,c_l$ (and also on the order in which these points are indexed).
Usually, it is clear from the context which points (and in which order) we use.

\section{Circles with points inside}\label{sec3}
In this section, we show that the existence of a betweenness isomorphism
between $S$ and $R$ (which we assume to be of the form described in 
\Cref{prob1x}) is closely related to the isomorphism of the corresponding group actions. 
This helps us to find necessary and sufficient conditions on $S$ and $R$ to be betweenness isomorphic.

For the rest of this section, we fix $l\in\mathbb N$ and two sets $S=C\cup\{c_1,\ldots,c_l\}$ and $R=D\cup\{d_1,\ldots,d_l\}$, where $C,D$ are 
circles in the plane, $c_1,\ldots,c_l$ are pairwise distinct points from the interior of $C$, and $d_1,\ldots,d_l$ are pairwise distinct points from 
the interior of $D$. Our aim is to describe when $S$ and $R$ are betweenness isomorphic.
	
In the following, we denote by $\alpha$, resp. $\beta$, the right action (described in the previous section) of the group $G_l$ on the set $C$, resp. $D$. We will use the short notation 
\[c\cdot g = \alpha(c,g)\quad\text{and}\quad d\cdot g=\beta(d,g),\]
for all $c\in C$, $d\in D$ and $g\in G_l$.

Recall that an \emph{isomorphism} of the two group actions $\alpha$ and $\beta$ is a pair $(\psi,\varphi)$, where $\psi$ is an automorphism of the group $G_l$ and $\varphi\colon C\to D$ is a bijection such that
\[\varphi(c\cdot g)=\varphi(c)\cdot\psi(g)\quad\text{for all }c\in C\text{ and }g\in G_l.\]
	
For a permutation $\sigma\colon\{1,\ldots,l\}\to\{1,\ldots,l\}$, let $\psi_\sigma$ be the automorphism of the group $G_l$ given by
\[\psi_\sigma(\emptyset)=\emptyset\text{ and }\psi_\sigma(i_1,\ldots,i_n)=(\sigma(i_1),\ldots,\sigma(i_n))\quad\text{for every }n\in\mathbb N\text{ and every }g=(i_1,\ldots,i_n)\in G_l.\] 

\begin{theorem}\label{t:characterization}
Let $f\colon S\to R$ be a bijection. Then $f$ is a betweenness isomorphism 
of $S$ and $R$ if and only if
\begin{enumerate}[label={\rm (\Roman*)}]
\item\label{fci} $f(\{c_1,\ldots,c_l\})=\{d_1,\ldots,d_l\}$,
\item\label{fch} $f|_{\chS(\{c_1,\ldots,c_l\})}$ is a 
betweenness isomorphism of $\chS(\{c_1,\ldots,c_l\})$ and 
$\chR(\{d_1,\ldots,d_l\})$,
\item\label{psif}
$(\psi_\sigma,f|_C)$ is an isomorphism of the group actions $\alpha$ and $\beta$,
where the permutation $\sigma$ is defined by $f(c_i)=d_{\sigma(i)}$ for every $i\in\{1,\ldots,l\}$.
\end{enumerate}
\end{theorem}

Before proving \Cref{t:characterization} we do a comment about condition~\ref{fci}.

\begin{remark}
Let us note that condition~\ref{fci} in \cref{t:characterization} follows from condition~\ref{fch}; so \cref{t:characterization} would hold even after deleting condition~\ref{fci} from its statement. This can be shown by an easy application of Lemma~\ref{l:ext-ch}\ref{ext} (and it will be explained in detail in \cref{r:Afterchara} \ref{extreme_points}). However, we will refer to~\ref{fci} many times, so we keep it in \cref{t:characterization}.
\end{remark}

\begin{proof}
Suppose first that $f\colon S\to R$ is a betweenness isomorphism of $S$ and $R$.

By assertion \ref{ext} of \Cref{l:ext-ch}, we have 
$f(C)=f(\ext(S))=\ext(R)=D$, and so \ref{fci} holds, 
and the permutation $\sigma$ from~\ref{psif} is defined correctly.
		
By assertion \ref{ch} of \Cref{l:ext-ch} and \ref{fci} 
(which we just proved), it holds
\begin{eqnarray*}
f(\chS(\{c_1,\ldots,c_l\}))=\chR(\{d_1,\ldots,d_l\}),
\end{eqnarray*}
and~\ref{fch} follows.

By~\ref{fci} (which we already verified), it follows that $f|_C$ is a bijection between $C$ and $D$. So, to verify~\ref{psif}, we must only show that $f(c\cdot g)=f(c)\cdot\psi_\sigma(g)$ for all $c\in C$ and $g\in G_l$.

First, fix $c\in C$ and $i\in\{1,\ldots,l\}$. 
As $f$ is a betweenness homomorphism and $c_i\in (c,c\cdot(i))$, it follows that
\begin{equation*}
d_{\sigma(i)}=f(c_i)\in(f(c),f(c\cdot(i))).
\end{equation*}
Simultaneously, it holds
\[d_{\sigma(i)}\in(f(c),f(c)\cdot(\sigma(i))).\]
As there is exactly one $d\in D$ such that $d_{\sigma(i)}\in(f(c),d)$, it follows that
\begin{eqnarray*}
f(c\cdot(i))=f(c)\cdot(\sigma(i))=f(c)\cdot\psi_\sigma(i).
\end{eqnarray*}
Now, a simple induction yields
\begin{equation*}
f(c\cdot(i_1,\ldots,i_n))=f(c)\cdot(\sigma(i_1),\ldots,\sigma(i_n))=f(c)\cdot\psi_\sigma(i_1,\ldots,i_n)
\end{equation*}
for every $n\in\mathbb N$ and every $g=(i_1,\ldots,i_n)\in G_l$, and the conclusion follows.
		
Now, suppose that~\ref{fci}, \ref{fch} and~\ref{psif} hold true. Let us fix $a,b,c\in S$. We must show that
\begin{equation*}
c\in(a,b)\iff f(c)\in(f(a),f(b)).
\end{equation*}
We distinguish four cases.
		
\textbf{Case 1:} $a,b,c\in\chS(\{c_1,\ldots,c_l\})$.
		
In this case, just apply~\ref{fch}.
		
\textbf{Case 2:} Exactly two of the points $a$, $b$, $c$ belong to $\chS(\{c_1,\ldots,c_l\})$.
		
As $f$ maps $\chS(\{c_1,\ldots,c_l\})$ onto $\chR(\{d_1,\ldots,d_l\})$ (by~\ref{fch}), exactly two of the points $f(a)$, $f(b)$, $f(c)$ belong to $\chR(\{d_1,\ldots,d_l\})$. By the definition of a collinear hull, it follows that $a,b,c$ are not collinear, and similarly $f(a)$, $f(b)$, $f(c)$ are not collinear. So $c\notin(a,b)$ and $f(c)\notin(f(a),f(b))$.
		
\textbf{Case 3:} $c\in C$.

In that case, $c\notin(a,b)$ as $c$ is extreme in $S$, and $f(c)\notin(f(a),f(b))$ as $f(c)$ is extreme in $R$, by assertion~\ref{ext} of \Cref{l:ext-ch}.
		
\textbf{Case 4:} $c\notin C$ and $a,b\in C$.
		
Then there is $i\in\{1,\ldots,l\}$ such that $c=c_i$. If $\sigma$ is as in~\ref{psif}, then
$f(c)=d_{\sigma(i)}$. So we have
\[c\in(a,b)\iff b=a\cdot(i)\stackrel{\ref{psif}}{\iff}f(b)=f(a)\cdot(\sigma(i))\iff f(c)=d_{\sigma(i)}\in(f(a),f(b)).\]
		
Finally, we note that the four cases above cover all possible situations.
\end{proof}
	
It is easy to see that the collinear hull $\chS(\{c_1,\ldots,c_l\})$ is the union of some (in fact, finitely many and pairwise disjoint) orbits given by the action $\alpha$. Similarly, $\chR(\{d_1,\ldots,d_l\})$ is the union of some orbits given by the action $\beta$. So we can define the right action $\alpha'$, resp. $\beta'$, of the group $G_l$ on the set $C\setminus\chS(\{c_1,\ldots,c_l\})$, resp. $D\setminus\chR(\{d_1,\ldots,d_l\})$, as the restriction of $\alpha$, resp. $\beta$, to the corresponding set. 

\begin{remark}\label{r:characterization}
\Cref{t:characterization} remains true if we replace condition \ref{psif} with 
\begin{enumerate}[label={\rm (III')}]
\item\label{psif'}
$(\psi_\sigma,f|_{C\setminus\chS(\{c_1,\ldots,c_l\})})$ is an isomorphism of the group actions $\alpha'$ and $\beta'$, where the permutation $\sigma$ is defined by $f(c_i)=d_{\sigma(i)}$ for every $i\in\{1,\ldots,l\}$.
\end{enumerate}
\end{remark}

\begin{proof}
Almost the same proof still works, only Case 4 would have to be restricted to `$c\notin C$ and $a,b\in C\setminus\chS(\{c_1,\ldots,c_l\})$' (then the four cases still cover all possible situations).
\end{proof}	

The following example shows that condition \ref{fch} in \Cref{t:characterization} cannot be replaced by
\begin{enumerate}[label={\rm (ii')}]
\item\label{fch'} $f|_{\{c_1,\ldots,c_l\}}$ is a betweenness 
isomorphism of $\{c_1,\ldots,c_l\}$ and $\{d_1,\ldots,d_l\}$.
\end{enumerate}

\begin{example}
Fix $l=2$ and let $R=S=C\cup\{c_1,c_2\}$. Denote by $x_1$ and $x_2$ the two points of $C$ that are collinear with $c_1$ and $c_2$. Define a map $f\colon S\to R$ putting
\[f(x_1)=x_2,\quad f(x_2)=x_1,\quad f(x)=x\text{ for every }x\in S\setminus\{x_1,x_2\}.\]
It is clear that $f$ is not a betweenness isomorphism. However, \ref{fci}, \ref{fch'}, and \ref{psif'} hold with $\sigma$ being the identity permutation. To see that also \ref{psif} holds (with the identity permutation $\sigma$) observe that for all $i,j\in\{1,2\}$ we have
\begin{equation*}
f(x_i\cdot(j))=f(x_{3-i})=x_i=x_{3-i}\cdot(j)=f(x_i)\cdot(j),
\end{equation*}
and by induction
\begin{equation*}
f(x_i\cdot g)=f(x_i)\cdot g\quad\text{for every }g\in G_2,
\end{equation*}
which is what we wanted to show.
\end{example}

From now on, put
\begin{equation*}
\mathcal O_C=\left\{\orb(c):c\in C\setminus\chS\{c_1,\ldots,c_l\}\right\},
\end{equation*}
and
\begin{equation*}
\mathcal O_D=\left\{\orb(d):d\in D\setminus\chR\{d_1,\ldots,d_l\}\right\}.
\end{equation*}

\begin{corollary}\label{c:characterization}
The sets $S$ and $R$ are betweenness isomorphic if and only if there 
exists a betweenness isomorphism  $\tilde f$ of 
$\chS(\{c_1,\ldots,c_l\})$ and $\chR(\{d_1,\ldots,d_l\})$ 
with the following property:
\begin{enumerate}[label={\rm (O)}]
\item\label{P} Let $\sigma\colon\{1,\ldots,l\}\to\{1,\ldots,l\}$ be the permutation that corresponds to $\tilde f$, i.e.,
\begin{equation}\label{tilde-fci}
\tilde f(c_i)=d_{\sigma(i)}\quad\text{for every }i\in\{1,\ldots,l\}.
\end{equation}
Then there exists a bijection $h\colon\mathcal O_C\to\mathcal O_D$ such that, for every $\orb\in\mathcal O_C$, there are $x_\orb\in \orb$ and $y_\orb\in h(\orb)$ with
\begin{equation}\label{eq:stabs}
\stab(y_\orb)=\left\{\psi_\sigma(g):g\in \textbf{S}(x_\orb)\right\}.
\end{equation}
\end{enumerate}		
\end{corollary}

Before proving the above result, let us explain two details of its formulation.

\begin{remark}\label{r:Afterchara}
\begin{enumerate}[label={\rm (\roman*)}]
\item\label{extreme_points} If $\tilde f$ is a betweenness isomorphism of $\chS(\{c_1,\ldots,c_l\})$ and $\chR(\{d_1,\ldots,d_l\})$, then $\tilde f$ maps $\{c_1,\ldots,c_l\}$ bijectively onto $\{d_1,\ldots,d_l\}$ so that the definition of permutation $\sigma$ in condition \ref{P} is correct.
Indeed, for $l=1$, we have $\{c_1\}=\chS(\{c_1\})$ and $\{d_1\}=\chS(\{d_1\})$.
For $l\ge 2$, $c_1,\ldots,c_l$ are exactly all non-extreme points in 
$\chS(\{c_1,\ldots,c_l\})$, and similarly, $d_1,\ldots,d_l$ are exactly all 
non-extreme points in $\chS(\{d_1,\ldots,d_l\})$. So it is enough to apply assertion \ref{ext} of \Cref{l:ext-ch}.
\item If $f\colon S\to R$ is a betweenness isomorphism, then the map $\tilde f$ occurring in \Cref{c:characterization} can be chosen as the restriction of $f$ to $\chS(\{c_1,\ldots,c_l\})$. 
This choice will be made in the proof of \cref{c:characterization}.
\end{enumerate}
\end{remark}
 
\begin{proof}[Proof of \Cref{c:characterization}]
Suppose first that there exists a betweenness isomorphism $f\colon S\to R$.
Then~\ref{fci}, \ref{fch} and~\ref{psif'} hold by \Cref{t:characterization} and \Cref{r:characterization}.
Let $\tilde f$ be the restriction of $f$ to $\chS(\{c_1,\ldots,c_l\})$.
Then, by~\ref{fch}, $\tilde f$ is a betweenness isomorphism
of $\chS(\{c_1,\ldots,c_l\})$ and $\chR(\{d_1,\ldots,d_l\})$. 

Let $\sigma$ be the permutation defined by \eqref{tilde-fci}; cf.\
assertion \cref{extreme_points} of \cref{r:Afterchara} for its existence.
From~\ref{psif'} we see that the $f$-image of every orbit given by the action $\alpha'$ is an orbit given by the action $\beta'$.
So we can define the bijection $h\colon\mathcal O_C\to\mathcal O_D$ by
\begin{equation*}
h(\orb)=\{f(c):c\in\orb\}.
\end{equation*}
Finally, for every $\orb\in\mathcal O_C$ we fix an arbitrary $x_\orb\in \orb$ and put $y_\orb=f(x_\orb)$.
Then
\begin{equation*}
g\in \stab(x_\orb)\iff x_\orb\cdot g=x_\orb\stackrel{\ref{psif'}}{\iff}f(x_\orb)\cdot\psi_\sigma(g)=f(x_\orb)\iff y_\orb\cdot\psi_\sigma(g)=y_\orb\iff\psi_\sigma(g)\in \stab(y_\orb),
\end{equation*}
and~\eqref{eq:stabs} follows.

Suppose now that there exists a betweenness isomorphism $\tilde f\colon \chS(\{c_1,\ldots,c_l\})\to\chR(\{d_1,\ldots,d_l\})$ as in the statement of the corollary, i.e., condition \ref{P} holds. 

We note that, by \cref{r:Afterchara}, 
$\tilde f$ maps $\{c_1,\ldots,c_l\}$ onto $\{d_1,\ldots,d_l\}$. 

We define $f\colon S\to R$ as an extension of $\tilde f$, so we only need to specify values of $f$ on
\begin{equation*}
C\setminus\chS(\{c_1,\ldots,c_l\})=\bigcup\mathcal O_C=\left\{x_\orb\cdot g:\orb\in\mathcal O_C,g\in G_l\right\}.
\end{equation*}
For every $\orb\in\mathcal O_C$ and every $g\in G_l$, we put
\begin{equation}\label{e:fdef}
f(x_\orb\cdot g)=y_\orb\cdot\psi_\sigma(g).
\end{equation}
To verify the correctness of this definition,
assume that $x_\orb\cdot g=x_{\orb'}\cdot g'$ for some $\orb,\orb'\in\mathcal O_C$ and $g,g'\in G_l$. Then $\orb=\orb'$, otherwise $x_\orb\cdot g$, $x_{\orb'}\cdot g'$ would belong to distinct (and hence disjoint) orbits. The equality $\orb=\orb'$ implies $x_\orb=x_{\orb'}$ and by our assumption we have $x_\orb \cdot g=x_\orb \cdot g'$, and hence $g(g')^{-1}\in \stab(x_\orb)$.
This together with~\eqref{eq:stabs} implies $\psi_\sigma(g(g')^{-1})\in \textbf{S}(y_\orb)$,
and so $y_\orb\cdot\psi_\sigma(g(g')^{-1})=y_\orb$.
Applying (from the right) the group action with $\psi_\sigma(g')$ we get
\begin{equation*}
y_\orb\cdot\psi_\sigma(g)=(y_\orb\cdot\psi_\sigma(g(g')^{-1}))\cdot\psi_\sigma(g')
=y_\orb\cdot\psi_\sigma(g'),	
\end{equation*}
and the correctness of the definition follows.

From~\eqref{e:fdef}, we easily deduce that $f(x) \in \orb(y_\orb) = h(\orb)$
for every $x\in \orb$, hence $f(\orb) \subset h(\orb)$.

To verify that $f$ is a betweenness isomorphism, we will apply \Cref{t:characterization} and \Cref{r:characterization}. 

We start by showing that $f$ is a bijection. As it is an extension of $\tilde f$, it is enough to show that the restriction of $f$ to $C\setminus\chS(\{c_1,\ldots,c_l\})$ is injective and onto $D\setminus\chR(\{d_1,\ldots,d_l\})$.
Suppose first that $f(x_\orb\cdot g)=f(x_{\orb'}\cdot g')$ for some $\orb,\orb'\in\mathcal O_C$ and $g,g'\in G_l$. Since $f(\orb)\subset h(\orb)$, $f(\orb')\subset h(\orb')$ and $h$ is a bijection, we must have $\orb=\orb'$, in particular $x_\orb=x_{\orb'}$ and $y_\orb=y_{\orb'}$.
In that case, we have $y_\orb\cdot\psi_\sigma(g)=y_\orb\cdot\psi_\sigma(g')$, and hence $\psi_\sigma(g(g')^{-1})\in \stab(y_\orb)$. This jointly with~\eqref{eq:stabs} imply $g(g')^{-1}\in \stab(x_\orb)$. So $x_\orb\cdot g=x_\orb\cdot g'$, and the injectivity of $f$ follows.
Now, fix some
\begin{equation*}
d\in D\setminus\chR(\{d_1,\ldots,d_l\})=\bigcup\mathcal O_D.
\end{equation*}
As $h$ is a bijection, there is $\orb\in\mathcal O_C$ such that $d\in h(\orb)$. So we can find $g\in G_l$ such that $d=y_\orb\cdot g$. Then $d=f(x_\orb\cdot\psi_\sigma^{-1}(g))$, and the surjectivity of $f$ follows.
		
We will check conditions~\ref{fci} and \ref{fch} from \Cref{t:characterization}, as well as condition \ref{psif'} from \Cref{r:characterization}.

Condition~\ref{fci} holds as $\tilde f$ maps $\{c_1,\ldots,c_l\}$ onto 
$\{d_1,\ldots,d_l\}$  by \cref{r:Afterchara}, and as $f$ is an extension of $\tilde f$.
Condition~\ref{fch} holds as $f$ is an extension of $\tilde f$.
So it only remains to check condition~\ref{psif'}.
Let us fix $c\in C\setminus\chS(\{c_1,\ldots,c_l\})$ and $g\in G_l$, we must show that
\begin{equation*}
f(c\cdot g)=f(c)\cdot\psi_\sigma(g).
\end{equation*}
Let $\orb=\orb(c)$ be the orbit of $c$, and let $g'\in G_l$ be such that $c=x_\orb\cdot g'$. Then 
\begin{equation*}
f(c\cdot g)=f(x_\orb\cdot g'g)=y_\orb\cdot\psi_\sigma(g'g)=(y_\orb\cdot\psi_\sigma(g'))\cdot\psi_\sigma(g)=f(x_\orb\cdot g')\cdot\psi_\sigma(g)=f(c)\cdot\psi_\sigma(g),
\end{equation*}
as we needed.
\end{proof}

\section{Circles with collinear points inside}\label{sec4}
We now provide some partial answers to \Cref{prob1x}, namely in the case when the $l$ points in the interior of each of the circles are collinear.

\subsection{The case $l=1$}
\label{subsec:l1}
This case is very easy.
We deal with it by making use of~\Cref{c:characterization}.
	
\begin{theorem}\label{t:l=1}
Let $S,R$ be sets of the form $S=C\cup\{c_1\}$ and $R=D\cup\{d_1\}$, where $C,D$ are circles in the plane, $c_1$ is a point from the interior of $C$, and $d_1$ is a point from the interior of $D$. Then the sets $S$ and $R$ are betweenness isomorphic.
\end{theorem}
	
\begin{proof}
Trivially, $\chS(\{c_1\})=\{c_1\}$ and $\chR(\{d_1\})=\{d_1\}$. So the only map $\tilde f\colon\chS(\{c_1\})\to\chR(\{d_1\})$ (given by $\tilde f(c_1)=d_1$) is obviously a betweenness isomorphism; the permutation $\sigma$ that corresponds to $\tilde f$ is the identity. 

In the considered case every orbit from $\mathcal O_C$ is a two point subset of $C$, every orbit from $\mathcal O_D$ is a two point subset of $D$. Therefore, the cardinality of both sets $\mathcal O_C$ and $\mathcal O_D$ is continuum; 
we fix a bijection $h$ between them arbitrarily. For every $\mathbf O\in\mathcal O_C$, we also fix $x_{\mathbf O}\in\mathbf O$ and $y_{\mathbf O}\in h(\mathbf O)$ arbitrarily.
To see that \eqref{eq:stabs} holds (where $\psi_\sigma$ is the identity automorphism) it suffices to observe that all the stabilizers $\stab(c)$, where $c\in C$, and $\stab(d)$, where $d\in D$, are equal to the trivial subgroup $\{\emptyset\}$. 

Now, the conclusion immediately follows from \Cref{c:characterization}.
\end{proof}

\subsection{The case $l=2$}\label{subsec:l2}
This is still easy with the application of~\Cref{c:characterization}.

Before formulating the result for the case $l=2$, we prove the following useful fact.

\begin{lemma}\label{l:uperLowerFirst}
Let $S$ be a set of the form $S=C\cup\{c_1,c_2\}$, where $C$ is 
a circle in the plane and $c_1,c_2$ are two distinct points from 
the interior of $C$. Let $c\in C\setminus\chS(\{c_1,c_2\})$ 
and $(i_1,\ldots,i_m)\in\{1,2\}^m$ with $m\in\mathbb N$.
\begin{enumerate}[label={\rm (\roman*)}]
\item If $m$ is odd, then 
\begin{equation}\label{e:Rcc}
(R_{i_m}\circ\dots \circ R_{i_1})(c)\neq c.
\end{equation}
\item If $m$ is even and $(i_1,\ldots,i_m)
\in G_2\setminus\{\emptyset\}$, then \eqref{e:Rcc} holds.
\end{enumerate}
\end{lemma}

\begin{proof}
Let $C_+$ denote the intersection of $C$ with one of the open half-planes given by the line passing through $c_1$ and $c_2$, and let $C_-$ denote the intersection of $C$ with the other open half-plane given by the same line.
Then $C\setminus\chS(\{c_1,c_2\})=C_+\cup C_-$.

(i) If $m$ is odd, then it suffices to note that $(R_{i_m}\circ\dots \circ R_{i_1})(C_+)\subset C_-$ and $(R_{i_m}\circ\dots\circ R_{i_1})(C_-)\subset C_+$.

(ii) If $m$ is even, then $(R_{i_m}\circ\dots \circ R_{i_1})(C_+)\subset C_+$ and $(R_{i_m}\circ\dots\circ R_{i_1})(C_-)\subset C_-$.

Note that if $g=(i_1,\ldots,i_m)\in G_2\setminus\{\emptyset\}$ then $g$, being an irreducible configuration of elements of $\{1,2\}$, is a finite (nonempty) sequence of alternating $1$'s and $2$'s, ending with $1$ or $2$. 
Therefore, there exists $n\in\mathbb N$ such that either $R_{i_m}\circ\dots\circ R_{i_1}=(R_1\circ R_2)^n$ or $R_{i_m}\circ\dots\circ R_{i_1}=(R_2\circ R_1)^n$.
Note that the maps $(R_2\circ R_1)^n$ and $(R_1\circ R_2)^n$ are mutual inverses.
Thus the statements $(R_1\circ R_2)^n(c)\neq c$ and $(R_2\circ R_1)^n(c)\neq c$ are equivalent.
So, to complete the proof, it is enough to show that $(R_1\circ R_2)^n(c)\neq c$.

Assume that $c\in C_+$; the case $c\in C_-$ is analogous. 

Clearly, $(R_1\circ R_2)(c)\neq c$; otherwise we would have $R_2(c)=R_1(c)$, which implies that $c_1$, $c_2$ both lie 
on the line passing through $c$ and $R_1(c)$, and this is only possible if $c_1=c_2$ or $c\in \chS(\{c_1,c_2\})$, contradicting our assumptions.

The point $c$ splits the arc $C_+$ into two open subarcs, say $C_+^1$ and $C_+^2$. 
Let $C_+^1$ be the subarc that contains $(R_1\circ R_2)(c)$. 
From a picture, the reader may verify that
\begin{equation*}
(R_1\circ R_2)(C_+^1)\subset C_+^1.
\end{equation*}
Consequently, by an easy induction,
\begin{equation*}
(R_1\circ R_2)^{k-1}(C_+^1)\subset C_+^1\quad\text{for every } k\in\mathbb N
\end{equation*}
(where we use the convention that $(R_1\circ R_2)^0$ is the identity map).
Hence
\begin{equation*}
(R_1\circ R_2)^n(c)=(R_1\circ R_2)^{n-1}((R_1\circ R_2)(c))\in(R_1\circ R_2)^{n-1}(C_+^1)\subset C_+^1.
\end{equation*}
As $c\notin C_+^1$, the conclusion follows.
\end{proof}

\begin{theorem}\label{t:l=2}
Let $S,R$ be sets of the form $S=C\cup\{c_1,c_2\}$ and $R=D\cup\{d_1,d_2\}$, where $C,D$ are circles in the plane, $c_1,c_2$ are distinct points from the interior of $C$, and $d_1,d_2$ are distinct points from the interior of $D$. Then the sets $S$ and $R$ are betweenness isomorphic.
\end{theorem}

\begin{proof}
In the case under consideration we have $\chS(\{c_1,c_2\})=\{c_1,c_2,x_1,x_2\}$, where $x_1,x_2\in C$ are distinct points, collinear with $c_1$ and $c_2$.
Similarly, $\chR(\{d_1,d_2\})$ consists of exactly four collinear points and two of them lying on the circle $D$. 
Therefore, it easily follows that there exists a betweenness isomorphism $\tilde f\colon\chS(\{c_1,c_2\})\to\chR(\{d_1,d_2\})$ such that $\tilde f(c_1)=d_1$ and $\tilde f(c_2)=d_2$; then the identity permutation on $\{1,2\}$ is the permutation $\sigma$ that corresponds to $\tilde f$ as in~\cref{tilde-fci}.

As all orbits $\orb \in \mathcal O_C$ and $\orb \in \mathcal O_D$ are obviously
at most
countable, the cardinality of each of the sets $\mathcal O_C$ and $\mathcal O_D$ is continuum.
We fix a bijection $h$ of $\mathcal O_C$ and $\mathcal O_D$. For every $\mathbf O\in\mathcal O_C$, we fix $x_{\mathbf O}\in\mathbf O$ and $y_{\mathbf O}\in h(\mathbf O)$ arbitrarily.
To complete the proof, by applying \Cref{c:characterization}, we only need to show two properties: $\stab(c)=\{\emptyset\}$ for every $c\in C\setminus\chS(\{c_1,c_2\})$, and $\stab(d)=\{\emptyset\}$ for every $d\in D\setminus\chR(\{d_1,d_2\})$.
The former property is equivalent to $c\cdot g\neq c$ for all $c\in C\setminus\chS(\{c_1,c_2\})$ and $g\in G_2\setminus\{\emptyset\}$, the latter property can be reformulated analogously.
So, it suffices to apply \Cref{l:uperLowerFirst}.
\end{proof}

\subsection{The case $l=3$}\label{subsec:l3}
This case is much more involved than that of $l\le 2$.
While we formulate some auxiliary results for general $l\ge 3$, our focus will be on the case $l=3$.

For $l\ge 3$, let $\mathcal K_l$ be the collection of all pairs $s=(C,(c_1,\ldots,c_l))$, where $C$ is a circle in the plane and $c_1,\ldots,c_l$ are pairwise distinct collinear points from the interior of $C$ such that either $c_1<c_2<\ldots<c_l$ or $c_l<c_{l-1}<\ldots<c_1$ in the natural order on the line passing through all these points. 

The main result of this section is that there are exactly countably many betweenness isomorphism classes of sets of the form $C\cup\{c_1,c_2,c_3\}$ where $(C,(c_1,c_2,c_3))\in\mathcal K_3$. Also, we will explicitly describe these isomorphism classes; see~\cref{t:isoclasses}.
This provides a partial answer to~\cref{prob1x} (for $3$ collinear points in the interior of the circle).

To formulate~\cref{t:isoclasses}, we start with some notation first. Let $l\ge 3$ and $g=(i_1,\ldots,i_n)\in G_l$ (we allow $n=0$ here; in that case, $g=\emptyset$). For every $i\in\{1,\ldots,l\}$, we put
\begin{equation*}
M_i=\{m\in\{1,\ldots,n\}:i_m=i\},
\end{equation*}
and
\begin{equation*}
N(g,i)=\sum_{m\in M_i}(-1)^{m+1};
\end{equation*}
in particular, $N(g,i)=0$ if $M_i=\emptyset$.
Then we define the \emph{signature} $N(g)$ of $g$ by
\begin{equation*}
N(g)=(N(g,1),\ldots,N(g,l))\in\mathbb Z^l.
\end{equation*}

We say that a vector $v=(v_1,\ldots,v_l)\in\mathbb Z^l$ is \emph{balanced} if $\sum_{i=1}^l v_i = 0$.

\begin{remark}\label{r:evenlenght}
Let $g\in G_l$. Then $\sum_{i=1}^lN(g,i)\in\{0,1\}$, and the signature $N(g)$ is a balanced vector (that is, $\sum_{i=1}^lN(g,i)=0$) if and only if (the irreducible sequence) $g$ is of even length.
\end{remark}

For a vector $v\in\mathbb Z^l\setminus \setof{\paren{0,\dots,0}}$, we denote by $\gcd(v)$ the greatest common divisor of all entries of $v$.

For every $s=(C,(c_1,\ldots,c_l))\in\mathcal K_l$, we define $\set(s)\subset\mathbb R^2$ by
\[\set(s)=C\cup\{c_1,\ldots,c_l\}.\]

Let $s=(C,(c_1,\ldots,c_l))\in\mathcal K_l$ be fixed.
We say that a vector $v\in\mathbb Z^l$ is a \emph{cycle} for $s$ if there are $g\in G_l$ and $c\in C\setminus\chu_{\set(s)}(\{c_1,\ldots,c_l\})$ such that $N(g)=v$ and $g\in\stab(c)$.
Later, we will see that, in this case, $g\in\stab(c)$ for every $g\in G_l$ with $N(g)=v$ and every $c\in C\setminus\chu_{\set(s)}(\{c_1,\ldots,c_l\})$ (see \cref{l:cyclePoint}).
This is the \emph{porism} phenomenon which makes the collinear case special, cf.~\cite{Kocik2013}.

For every $l\ge 3$ and every $v\in\mathbb Z^l$, we define
\begin{equation*}
[v]=\left\{s\in\mathcal K_l:v\text{ is a cycle for }s\right\},
\end{equation*}
and
\begin{equation*}
[[v]]=\left\{\set(s):s\in[v]\right\}.
\end{equation*}

Let $s\in\mathcal K_l$ be given.
Note that, while $\set(s)$ is a subset of the plane, $s$ itself is just a code for this subset.
If $l=3$, then each set of the form $C\cup\{c_1,c_2,c_3\}$, where $C$ is a circle in the plane and $c_1,c_2,c_3$ are pairwise distinct collinear points in the interior of $C$, has exactly two codes from $\mathcal K_3$.
If $c_2$ lies on the line segment connecting $c_1$ and $c_3$ (the two other cases are completely analogous), these codes are $(C,(c_1,c_2,c_3))$ and $(C,(c_3,c_2,c_1))$.

Now we are ready to formulate the main result of this section.

\begin{theorem}\label{t:isoclasses}
The following are all betweenness isomorphism classes of sets of the form $C\cup\{c_1,c_2,c_3\}$, where $C$ is a circle in the plane and $c_1,c_2,c_3$ are pairwise distinct collinear points in the interior of $C$:
\begin{align}
\label{eq:firstline}
&[[v]] \qquad \text{where }
v=(v_1,v_2,v_3)\in \mathbb Z^3\text{ is balanced with }
v_2 \ge 1, \ v_1 \le v_3 \le -1,\ \gcd(v) = 1,\\
\label{eq:secondline}
&\mathbb O \coloneqq
\left\{\set (s):s\in\mathcal K_3 \text{ is such that $(0,0,0)$ is the only cycle for }s\right\}.
\end{align}
Each of the classes is listed just once.
\end{theorem}

Let us recall the meaning of $[[v]]$ in~\eqref{eq:firstline}. Namely, if $v_1, v_3$ are negative integers with no common divisor, we let $v_2=\absof{v_1}+\absof{v_3}$ and $m=2 v_2$.
Then $[[v]]$ in~\eqref{eq:firstline} is the class of all sets $M\subset 
\mathbb R^2$ which may be written in the form $M=C\cup\{c_1,c_2,c_3\}$ where $c_1,c_2,c_3$ are
distinct collinear points in the interior of a circle $C$, such that there are
\begin{itemize}
    \item $c\in C$ not collinear with $c_1,c_2,c_3$,
    \item $(i_1,\dots,i_m)$ so that $i_j=2$ for all $j$ odd while on the even positions, $1$ appears $\absof{v_1}$ times and $3$ appears $\absof{v_3}$ times,
\end{itemize}
satisfying
\begin{equation}\label{e:CYCLE}   
\paren{P_{c_{i_m}}\circ P_{c_{i_{m-1}}}\circ\dots\circ P_{c_{i_2}} \circ P_{c_{i_1}}}(c)=c,
\end{equation}
where $P_{X}$ is the reversion map (through the point $X$) as defined in \cref{subsec:reversionETC}. 
(With the reference to \cref{l:cyclePoint} again, we note that if there are $c$ and $(i_1,\ldots,i_m)$ with the above property, then~\eqref{e:CYCLE} is actually true for all $c$ and $(i_1,\ldots,i_m)$ with the same property.)

\smallskip %

The rest of this section is devoted to the proof of~\cref{t:isoclasses}.

\begin{lemma}\label{l:even}
Let $l\ge 3$, $s=(C,(c_1,\ldots,c_l))\in\mathcal K_l$, $c\in C\setminus\chu_{\set(s)}(\{c_1,\ldots,c_l\})$ and $g\in G_l$ be given.
If $c\cdot g=c$ then the irreducible sequence $g$ is of even length.
\end{lemma}

\begin{proof}
If $g$ is of odd length then the points $c$ and $c\cdot g$ lie in different open half-planes given by the line passing through the points $c_1,\ldots,c_l$. In particular, $c\neq c\cdot g$.
\end{proof}

\begin{corollary}\label{cor:balanced}
Let $l\ge 3$ and $v\in\mathbb Z^l$ be such that $v$ is a cycle for some $s=(C,(c_1,\ldots,c_l))\in\mathcal K_l$.
Then $v$ is balanced.
\end{corollary}
\begin{proof}
Find $s=(C,(c_1,\ldots,c_l))\in\mathcal K_l$
such that $v$ is a cycle for $s$.
Then there are $g\in G_l$ and $c\in C\setminus\chu_{\set(s)}(\{c_1,\ldots,c_l\})$ such that $N(g)=v$ and $g\in\stab(c)$.
By \cref{l:even}, $g$ is of even length.
Consequently, $v=N(g)$ is balanced by \Cref{r:evenlenght}.
\end{proof}

We say that a permutation $(\pi_1,\ldots,\pi_n)$ is \emph{admissible} if, for every $i\in\{1,\ldots,n\}$, the parity of $\pi_i$ is the same as the parity of $i$, i.e., $(-1)^{\pi_i}=(-1)^i$.

Let $l\ge 3$ be given. To every $g\in G_l$ we will assign $\overline g\in G_l$ with the following properties:
\begin{itemize}
\item for each $i\in\{1,\ldots,l\}$, $\overline g$ contains $i$ only at even positions, or only at odd positions (if any at all),
\item the elements of $\overline g$, separately on even positions and on odd positions, form a non-decreasing sequence,
\item $N(g)=N(\overline g)$,
\item $\overline g$ is obtained from $g$ only with the use of admissible permutations and reductions.
\end{itemize}

We proceed as follows.
First, using an admissible permutation of elements of $g$, we shift all instances of $l$ to the left (obeying even/odd positions,
while respecting the order of other elements on even positions, and those on odd positions). 
Then, we reduce the obtained sequence.
Observe that, in this case, the reduction simply means a `removal' of all pairs $(l,l)$ from the start
of the sequence.
We obtain a new sequence $g^1$ which contains $l$ only 
at even positions, or only at odd positions.
This finishes the first step of our construction. In the next steps, we consecutively repeat this procedure with elements of $\{1,\ldots,l-1\}$ in the reverse order, namely with $l+1-i$ in the $i$-th step.
Let $\overline g$ be the sequence obtained after the last (that is, after the $l$th) iteration.
Note that the elements of $\overline g$ on even positions got 
sorted during our procedure to a non-decreasing sequence, the same for the elements on odd positions.
As we only used admissible permutations followed by reductions, it easily follows that $\overline g$ is an irreducible sequence which has all the required properties.
Indeed, it is easy to see that admissible permutations followed by reductions
do not change the signature $N(\cdot)$.

The following lemma is almost obvious so we omit the proof.
Recall that $N(\overline g)=N(g)$
and observe that $N(g)$ determines the number of occurrences of
each element from $\{1,\ldots,l\}$ in the sequence $\overline g$, and its placement
to either even or odd positions.

\begin{lemma}\label{fact:signatures}
Let $l\ge 3$ and $g_1,g_2\in G_l$ be given such that $N(g_1)=N(g_2)$.
Then $\overline{g_1}=\overline{g_2}$.
\end{lemma}

\begin{lemma}\label{l:admi+redu}
Let $l\ge 3$ and $s=(C,(c_1,\ldots,c_l))\in\mathcal K_l$
be given.
Let 
$\paren{i_1,\dots,i_n} \in \setof{1,\dots,l}^n$  with $n\in\mathbb N$.
\begin{enumerate}[label={\rm (\roman*)}]
\item If $\pi$ is an admissible permutation of\/ $\setof{1,\dots,l}$
then $R_{i_{\pi(n)}}\circ \cdots \circ R_{i_{\pi(1)}}
=R_{i_n}\circ \cdots \circ R_{i_1}$.
\item
If $m\in \mathbb N$ and $\paren{j_1,\dots,j_m}$
is obtained from $\paren{i_1,\dots,i_n}$
by reduction, then $R_{j_m}\circ \cdots \circ R_{j_1} =
R_{i_n}\circ \cdots \circ R_{i_1}$.
\end{enumerate}
\end{lemma}

\begin{proof}
As reversions are involutive, assertion (ii) follows.
To prove assertion (i) let us first assume that $i,j,k\in\{1,\ldots,l\}$
are
given.
By~\cite[Theorem~4]{Kocik2013}, the composition $R_{i}\circ R_{j}\circ R_{k}$ is a reversion (through some point $X\in\mathbb R^2\setminus C$).
As every reversion is
an involution, we have
\begin{equation*}
(R_{k}\circ R_{j}\circ R_{i})\circ (R_{k}\circ R_{j}\circ R_{i})= \mathrm{Id}_C,
\end{equation*}
where $\mathrm{Id}_C$ is the identity map on the circle $C$.
Multiplying by $R_{i}\circ R_{j}\circ R_{k}$
from the right hand side, we obtain
\begin{equation*}
R_{k}\circ R_{j}\circ R_{i}=R_{i}\circ R_{j}\circ R_{k}.
\end{equation*}

From this we see that interchanging of
two entries of a sequence on consecutive even positions
does not change
the resulting reversion composition map;
the same is true for consecutive odd positions.
As any permutation can be obtained by a sequence of swapping of neighbouring elements, we reach the conclusion for all admissible permutations.
\end{proof}

\begin{lemma}\label{l:theSameSignatures}
Let $l\ge 3$, $s=(C,(c_1,\ldots,c_l))\in\mathcal K_l$ and $c\in C$ be given.
Let $g_1,g_2\in G_l$ be such that $N(g_1)=N(g_2)$.
Then
\begin{equation*}
g_1\in\stab(c)\iff g_2\in\stab(c).
\end{equation*}
\end{lemma}

\begin{proof} 
Since $\overline g$ is obtained by admissible permutations and reductions, we have $c\cdot g=c\cdot \overline g$
for any $g\in G_l$
by \cref{l:admi+redu}.
As $N(g_1)=N(g_2)$, we have $\overline g_1 = \overline g_2$
by \cref{fact:signatures}.
Hence
$c\cdot g_1 = c\cdot \overline g_1 = c \cdot \overline g_2=
c\cdot g_2$.
Therefore $g_1\in \stab(c)$ if and only if $g_2\in\stab(c)$.
\end{proof}

\begin{lemma}\label{l:cyclePoint}
Let $s=(C,(c_1,\ldots,c_l))\in\mathcal K_l$ and $v\in\mathbb Z^l$ be fixed.
Suppose that $v$ is a cycle for $s$.
Then $g\in\stab(c)$ for every $g\in G_l$ with $N(g)=v$ and every $c\in C\setminus\chu_{\set(s)}(\{c_1,\ldots,c_l\})$.
\end{lemma}

\begin{proof}
As $v$ is a cycle for $s$, there are $g'\in G_l$ and $c'\in C\setminus\chu_{\set(s)}(\{c_1,\ldots,c_l\})$ such that $N(g')=v$ and $g'\in\stab(c')$.
Fix arbitrary $g\in G_l$ with $N(g)=v$.
By \cref{l:theSameSignatures}, as $N(g')=v=N(g)$, we have $g\in\stab(c')$.
So, to complete the proof, it is enough to show that
$\stab(c)=\stab(c')$ for every $c,c'\in C\setminus\chu_{\set(s)}(\{c_1,\ldots,c_l\})$.
Recall that if $g\in\stab(c)$ for some $c\in C\setminus\chS(\{c_1,\ldots,c_l\})$ then $g$ is of even length by \cref{l:even}.
So it suffices to apply ~\cite[Theorem~7]{Kocik2013}\footnote{Note that there is a typo in the statement of Theorem~7 in \cite{Kocik2013}. Namely, the premise `$\exists X\in K,\;\mathbf P(X)=X$' of the implication in~\cite[condition~(11)]{Kocik2013} should correctly state `$\exists X\in K\setminus L,\;\mathbf P(X)=X$', where $L$ is the line passing through the (collinear) points $\mathbf P_1,\ldots,\mathbf P_{2n}$; cf.\cite{Bogomolny1997}}.
\end{proof}

\begin{lemma}\label{l:plusMinus}
Let $l\ge 3$, $s=(C,(c_1,\ldots,c_l))\in\mathcal K_l$ and $v^1,v^2\in\mathbb Z^l$ be given such that $v^1,v^2$ are cycles for $s$.
Then $-v^1$, $v^1+v^2$ and $v^1-v^2$ are also cycles for $s$.
\end{lemma}

\begin{proof} By assumptions there are $g_1\in G_l$ and $c\in C\setminus\chu_{\set(s)}(\{c_1,\ldots,c_l\})$ such that $N(g_1)=v^1$ and $g_1\in\stab(c)$. Then $g_1^{-1}\in\stab(c)$.
By \cref{l:even}, the sequence $g_1$ is of even length. Hence we easily deduce
that $N(g_1^{-1})=-v^1$. 
So, $g_1^{-1}$ and $c$ witness that $-v^1$ is a cycle for $s$.

Let $g_2\in G_l$ be such that $N(g_2)=v^2$.
By \cref{l:cyclePoint}, we have $g_2\in\stab(c)$.
Clearly, $N(g_1g_2)=v^1+v^2$ (here, we use again the fact that $g_1$ is of even length) and $c\cdot g_1g_2=c\cdot g_2=c$.
So, $g_1g_2$ and $c$ witness that $v^1+v^2$ is a cycle for $s$.

To see that $v^1-v^2$ is a cycle for $s$ it suffices to note that $v^1-v^2=v^1+(-v^2)$ and apply what we already proved.
\end{proof}

\begin{lemma}\label{l:multiple}
Let $l\ge 3$, $s=(C,(c_1,\ldots,c_l))\in\mathcal K_l$, $v\in\mathbb Z^l$, and $k\in\mathbb Z\setminus\{0\}$ be given.
Then $v$ is a cycle for $s$ if and only if $kv$ is a cycle for $s$.
\end{lemma}

\begin{proof}
Without loss of generality, we may assume that $k>0$ because, by \cref{l:plusMinus}, $kv$ is a cycle for $s$ if and only if $-kv$ is a cycle for $s$.

For $k=1$, there is nothing to prove. So, we assume that $k\geq 2$.

Suppose first that $v$ is a cycle for $s$.
Then there are $g\in G_l$ and $c\in C\setminus\chu_{\set(s)}(\{c_1,\ldots,c_l\})$ such that $N(g)=v$ and $g\in\stab(c)$;
the sequence $g$ is of even length by \cref{l:even}.
So $g^k\in G_l$ (the concatenation of $k$ copies of $g$, reduced if necessary) clearly satisfies $N(g^k)=kN(g)=kv$.
Also, $g^k\in\stab(c)$ (as $g\in\stab(c)$ and each stabilizer is a subgroup).
Thus $kv$ is a cycle for $s$.

Now assume that $kv$ is a cycle for $s$.
Then there are $g\in G_l$ and $c\in C\setminus\chu_{\set(s)}(\{c_1,\ldots,c_l\})$ such that $N(g)=kv$ and $g\in\stab(c)$.
Recall that $N(\overline g)=N(g)$. So, by \cref{l:theSameSignatures}, we may assume without any loss of generality that $g=\overline g$.
In particular, for each $i\in\{1,\ldots,l\}$, $g=\overline{g}$ contains $i$ only at even positions, or only at odd positions. Also, as the signature of $g=\overline{g}$ is divisible by $k$, each $i\in\{1,\ldots,l\}$ occurs at a multiple of $k$ positions in $g$.
As $kv$ is a cycle, it is balanced by \cref{cor:balanced}; so $v$ is balanced, too.
This implies that $\sum\absof{v_i}$
is an even number.
Fix some $h\in G_l$ such that, for every $i\in\{1,\ldots,l\}$,
\begin{itemize}
\item the sequence $h$ has $i$ exactly on $|v_i|$ positions,
\item if $g=\overline g$ contains $i$ only at even positions, then $h$ also contains $i$ only at even positions,
\item if $g=\overline g$ contains $i$ only at odd positions, then $h$ also contains $i$ only at odd positions.
\end{itemize}
Then $h$ is of even length (as the length of $h$ equals $\sum\absof{v_i}$).
Hence it is clear that $h^k$ is an irreducible sequence which can be obtained from $g=\overline g$ by an admissible permutation.
As admissible permutations do not change signatures, it still holds
\begin{equation}\label{e:poradToPlati}
h^k\in\stab(c),
\end{equation}
by \cref{l:theSameSignatures}.

As $h$ is of even length, we have $kN(h)=N(h^k)$ and 
\begin{equation}\label{e:h}
h=(i_1,\ldots,i_{2n})
\end{equation}
for some non-negative $n\in\mathbb Z$. We see that
\begin{equation*}
kN(h)=N(h^k)=N(\overline g)=N(g)=kv.
\end{equation*}
In particular,
\begin{equation*}
N(h)=v.
\end{equation*}

If $n=0$ then $h=\emptyset$, and so $v=N(h)=(0,\ldots,0)\in\mathbb Z^l$.
In this case, $\emptyset\in G_l$ and any $c\in C\setminus\chu_{\set(s)}(\{c_1,\ldots,c_l\})$ trivially witness that $v$ is a cycle.
So we may assume that $n>0$.

By~\cite[Theroem~3]{Kocik2013}, for every $j_1,j_2,j_3\in\{1,\ldots,l\}$, the composition
$R_{j_3}\circ R_{j_2}\circ R_{j_1}$ is a reversion through some point lying in the interior of the circle $C$ and collinear with points $c_1,\ldots,c_l$.
One can easily conclude by induction on $m$ that, for every 
$j_1,\ldots,j_{2m-1}\in\{1,\ldots,l\}$, the composition 
$R_{j_{2m-1}}\circ\cdots\circ R_{j_1}$ is a reversion through some point, which lies 
in the interior of the circle $C$ and which is collinear with $c_1,\ldots,c_l$.
It follows that there is some $Y$ in the interior of $C$ such that, denoting by $P_Y$ the reversion map through $Y$, we have
\begin{equation}\label{e:popisReversema}
R_{i_{2n}}\circ\cdots\circ R_{i_1}=R_{i_{2n}}\circ P_Y,
\end{equation}
where $i_1,\ldots,i_{2n}$ are given by~\eqref{e:h}.
Consequently, it holds
\begin{equation}\label{e:dveReverze}
c\stackrel{\eqref{e:poradToPlati}}{=}c\cdot h^k\stackrel{\eqref{e:popisReversema}}{=}(R_{i_{2n}}\circ P_Y)^k(c).
\end{equation}
However, by \cref{l:uperLowerFirst}, this is possible only if $Y=c_{i_{2n}}$.
Then
$R_{i_{2n}}\circ P_Y$ is the identity map on the circle $C$.
Consequently, we have
\[c\cdot h\stackrel{\eqref{e:h}}{=}(R_{i_{2n}}\circ\cdots\circ R_{i_1})(c)\stackrel{\eqref{e:popisReversema}}{=}(R_{i_{2n}}\circ P_Y)(c)=c,\]
and so $h$ and $c$ witness that $v$ is a cycle for $s$.
\end{proof}

We need the following generalization of \cref{l:uperLowerFirst}.

\begin{lemma}\label{l:uperLowerFirst3}
Let $C$ be a circle in the plane and $c_1,c_2,c_3$ be pairwise distinct points from the interior of $C$, such that $c_2$ lies on the line segment connecting $c_1$ and $c_3$ (so, in particular, the points $c_1,c_2,c_3$ are collinear).
Let $c\in C\setminus\chS(\{c_1,c_2,c_3\})$ and $(i_1,\ldots,i_m)\in\{1,2,3\}^m$ with $m\in\mathbb N$.
Suppose that the sequence $(i_1,\ldots,i_m)$ has the elements of\/ $\{1,2\}$ on even positions and the $3$'s on odd positions, or vice versa. Then
\begin{equation}%
(R_{i_m}\circ\dots \circ R_{i_1})(c)\neq c.
\end{equation}
\end{lemma}

\begin{proof}   
Let $L$ denote the line passing through $c_1,c_2,c_3$, and $C_+$, $C_-$ the two connected components of $C\setminus L$. 
Then $C\setminus\chS(\{c_1,c_2,c_3\})=C_+\cup C_-$.
    
(i) If $m$ is odd, then it suffices to note that $(R_{i_m}\circ\dots \circ R_{i_1})(C_+)\subset C_-$ and $(R_{i_m}\circ\dots\circ R_{i_1})(C_-)\subset C_+$.

(ii) If $m$ is even, then $(R_{i_m}\circ\dots \circ R_{i_1})(C_+)\subset C_+$ and $(R_{i_m}\circ\dots\circ R_{i_1})(C_-)\subset C_-$.

Note that the maps $R_{i_m}\circ\dots \circ R_{i_1}$ and $R_{i_1}\circ\dots \circ R_{i_m}$ are mutual inverses.
In particular, the statements $(R_{i_m}\circ\dots \circ R_{i_1})(c)\neq c$ and $(R_{i_1}\circ\dots \circ R_{i_m})(c)\neq c$ are equivalent.
So we may assume that the sequence $(i_1,\ldots,i_m)$ has elements of $\{1,2\}$ on even positions and $3$'s on odd positions (as, in the `vice versa' case, we may replace the sequence $(i_1,\ldots,i_m)$ by $(i_m,\ldots,i_1)$).

Assume that $c\in C_+$; the case $c\in C_-$ is analogous. 

Clearly,
\begin{equation*}
(R_{i_2}\circ R_{i_1})(c)=(R_{i_2}\circ R_3)(c)\neq c.
\end{equation*}
The point $c$ splits the arc $C_+$ into two open subarcs, say $C_+^1$ and $C_+^2$. 
Let $C_+^1$ be the subarc that contains $(R_{i_2}\circ R_3)(c)$. 
From a picture, the reader may verify that
\begin{equation*}
(R_i\circ R_3)(C_+^1)\subset C_+^1\quad\text{whenever $i\in\{1,2\}$};
\end{equation*}
in particular,
\begin{equation*}
(R_{i_j}\circ R_{i_{j-1}})(C_+^1)=(R_{i_j}\circ R_3)(C_+^1)\subset C_+^1\quad\text{whenever $j\in\{2,\ldots,m\}$ is even}.
\end{equation*}
Consequently, by an easy induction,
\begin{equation*}
(R_{i_k}\circ\ldots\circ R_{i_3})(C_+^1)\subset C_+^1\quad \text{for every even }2<k\le m,
\end{equation*}
where we may also allow $k=2$ by interpreting the empty composition as the identity function (indeed, $C_+^1\subset C_+^1$). Hence
\begin{equation*}
(R_{i_m}\circ\ldots\circ R_{i_1})(c)=(R_{i_m}\circ\ldots\circ R_{i_3})(R_{i_2}\circ R_3(c))\in(R_{i_m}\circ\ldots\circ R_{i_3})(C_+^1)\subset C_+^1.
\end{equation*}
As $c\notin C_+^1$, the conclusion follows.
\end{proof}

\begin{lemma}\label{l:charNonempty}
Let $v=(v_1,v_2,v_3)\in\mathbb Z^3$ be given.
Then $[v]\neq\emptyset$ if and only if either $v_1v_2v_3\neq 0$, $|v_2|=|v_1|+|v_3|$ and $v$ is balanced, or $v=(0,0,0)$.
\end{lemma}

\begin{proof}
To prove one of the implications, assume that $[v]\neq\emptyset$ and $v\neq(0,0,0)$.
Then there is $s=(C,(c_1,c_2,c_3))\in\mathcal K_3$ such that $v$ is a cycle for $s$.
That is, there are $g\in G_3$ and $c\in C\setminus\chu_{\set(s)}(\{c_1,c_2,c_3\})$ such that $N(g)=v$ and $g\in\stab(c)$.
By \cref{l:theSameSignatures}, we may assume that $g=\overline g$.

First, we show that $v_1v_2v_3\neq 0$.
Suppose for a contradiction that $v_3=0$ (the cases $v_1=0$ and $v_2=0$ are completely analogous).
Then $g=\overline g$ consists only of (alternating) $1$'s and $2$'s.
Then, by \cref{l:uperLowerFirst},
the fact that $c=c\cdot g$ implies that $g=\emptyset$. But then $v=N(g)=(0,0,0)$, a contradiction.

By \cref{cor:balanced}, the vector $v$ is balanced, that is, $v_1+v_2+v_3=0$.
So we have, depending on the signs, $|v_2|=|v_1|+|v_3|$, $|v_1|=|v_2|+|v_3|$, or $|v_3|=|v_1|+|v_2|$.
It only remains to show that the latter two cases cannot happen.
We will only show that $|v_3|=|v_1|+|v_2|$ cannot happen; the other case is analogous.
So suppose for a contradiction that $|v_3|=|v_1|+|v_2|$.
Then the sequence $g=\overline g$ has the elements of $\{1,2\}$ on even positions and the $3$'s on odd positions, or vice versa.
Recall that $c_2$ lies on the line segment connecting $c_1$ and $c_3$ by the definition of $\mathcal K_3$.
Hence, by \cref{l:uperLowerFirst3}, $c\cdot g\neq c$, a contradiction.

Now we will show the other implication.
If $v=(0,0,0)$ then $[v]\neq\emptyset$ trivially ($(0,0,0)$ is a cycle for any $s=(C,(c_1,c_2,c_3))\in\mathcal K_3$, as witnessed by $g=\emptyset\in G_3$ and any $c\in C\setminus\chu_{\set(s)}(\{c_1,c_2,c_3\})$).
So assume that $v_1v_2v_3\neq 0$, $|v_2|=|v_1|+|v_3|$ and $v$ is balanced. 
Let $K=\{(x,y)\in\mathbb R^2:x^2+y^2=1\}$ be the unit circle in the plane.
We put $c_1=\left(-\frac12,0\right)$ and $c_3=\left(\frac12,0\right)$.
It is enough to show that there exists $a\in\left(-\frac12,\frac12\right)$ such that, for $c_2=\left(a,0\right)$, we have $(K,(c_1,c_2,c_3))\in[v]$.

In view of \Cref{l:plusMinus} we may (and do) assume that $v_2>0$.
Then, as $v$ is balanced, we have $v_1<0$ and $v_3<0$.
We put $g=(2,1)^{|v_1|}(2,3)^{|v_3|}\in G_3$.
As
\begin{equation*}
N(g)=(-|v_1|,|v_1|+|v_3|,-|v_3|)=v,	
\end{equation*}
it suffices to find $a\in\left(-\frac12,\frac12\right)$
such that, for $c_2=(a,0)$, 
it holds
\begin{equation}\label{e:weWant}
(0,1)\cdot g=(0,1).
\end{equation}
To shorten the notation in the proof, for every $a\in(-1,1)$, we will use the symbol $R_a$ for the reversion $P_{(a,0)}\colon K\to K$ through the point $(a,0)$.
Then~\eqref{e:weWant} translates to
\begin{equation}\label{e:X_a}
\left(\left(R_{\frac12}\circ R_a\right)^{|v_3|}\circ\left(R_{-\frac12}\circ R_a\right)^{|v_1|}\right)(0,1)=(0,1).
\end{equation}

Let us consider the map $\rho\colon\left[-\frac12,\frac12\right]\to K$ given by
\begin{equation*}
\rho(a)=\left(\left(R_{\frac12}\circ R_a\right)^{|v_3|}\circ\left(R_{-\frac12}\circ R_a\right)^{|v_1|}\right)(0,1).	
\end{equation*}
Note that
\begin{equation}\label{e:K++0}
\rho\left(-\frac12\right)=\left(R_{\frac12}\circ R_{-\frac12}\right)^{|v_3|}(0,1)\quad\text{and}\quad\rho\left(\frac12\right)=\left(R_{-\frac12}\circ R_{\frac12}\right)^{|v_1|}(0,1).
\end{equation}
We put 
\begin{equation*}
K_{+}=\left\{(x,y)\in K:x>0,y>0\right\}\quad\text{and}\quad K_{-}=\left\{(x,y)\in K:x<0,y>0\right\};
\end{equation*}
It is clear from a picture that 
\begin{equation}\label{e:K++1}
\left(R_{\frac12}\circ R_{-\frac12}\right)(0,1)\in K_{+},\quad \left(R_{-\frac12}\circ R_{\frac12}\right)(0,1)\in K_{-},
\end{equation}
and
\begin{equation}\label{e:K++2}
\left(R_{\frac12}\circ R_{-\frac12}\right)(K_{+})\subset K_{+},\quad \left(R_{-\frac12}\circ R_{\frac12}\right)(K_{-})\subset K_{-}.
\end{equation}
By~\eqref{e:K++0}, \eqref{e:K++1} and~\eqref{e:K++2}, as $v_1v_3\neq 0$, we obtain that
\begin{equation}\label{e:K++3}
\rho\left(-\frac12\right)\in K_{+}\quad\text{and}\quad\rho\left(\frac12\right)\in K_{-}.
\end{equation}

To complete the proof, it only remains to show that the map $\rho$ is continuous.
Indeed, continuity would imply the existence of some $a\in\left(-\frac12,\frac12\right)$ such that the first coordinate of $\rho(a)$ is zero (as the first coordinates of $\rho\left(-\frac12\right)$ and $\rho\left(\frac12\right)$ have different signs by~\eqref{e:K++3}).
But the image of $\rho$ is clearly contained in the intersection of $K$ with the upper half plane, so this would mean that $\rho(a)=(0,1)$, which is equivalent to~\eqref{e:X_a}.

Next, we verify that $\rho$ is continuous.  We claim that, for every $n\in\mathbb N$, the map $\Psi_n\colon K\times\left[-\frac12,\frac12\right]^n\to K$ defined by
\begin{equation*}
\Psi_n(c,(a_1,\ldots,a_n))=\left(R_{a_n}\circ\ldots\circ R_{a_1}\right)(c)
\end{equation*}
is continuous.
This is all we need as the map $\rho$ is the composition of 
$\Psi_{|v_1|+|v_2|+|v_3|}$ and the (linear) continuous map
\begin{equation*}
\left[-\frac12,\frac12\right]\ni a\mapsto \left( (0,1),\left(a, -\frac12,a, -\frac12,
\dots,a, -\frac12, a, \frac12,a, \frac12, \dots,
a, \frac12\right) \right)\in K\times\left[-\frac12,\frac12\right]^{|v_1|+|v_2|+|v_3|},
\end{equation*}
where $-\frac12$ appears $\abs{v_1}$ times and $\frac12$ appears
$\abs{v_3}$ times.

We prove our claim by induction on $n$.
For $n=1$, this is obvious from the geometrical meaning of the definition of the reversion.
Now fix $n\in\mathbb N$ and suppose that $\Psi_i$ is continuous for every $i\in\{1,\ldots,n\}$.
For every $(c,(a_1,\ldots,a_{n+1}))\in K\times\left[-\frac12,\frac12\right]^{n+1}$, we have
\begin{equation*}
\Psi_{n+1}(c,(a_1,\ldots,a_{n+1}))=R_{a_{n+1}}(\Psi_n(c,(a_1,\ldots,a_n)))=\Psi_1(\Psi_n(c,(a_1,\ldots,a_n),a_{n+1})).
\end{equation*}
So, the continuity of $\Psi_{n+1}$ follows by the continuity of $\Psi_1$ and $\Psi_n$.
\end{proof}

\begin{lemma}\label{l:dependency}
Let $v^1,v^2\in\mathbb Z^3$ be given.
Then $[v^1]\cap[v^2]\neq\emptyset$ if and only if $[v^1]\neq\emptyset$, $[v^2]\neq\emptyset$ and $v^1,v^2$ are linearly dependent.
\end{lemma}

\begin{proof}
Suppose first that $[v^1]\neq\emptyset$, $[v^2]\neq\emptyset$ and $v^1,v^2$ are linearly dependent.
If $v^1=(0,0,0)$ then, trivially, $[v^1]=\mathcal K_3$.
In that case, $[v^1]\cap[v^2]=[v^2]\neq\emptyset$.
Likewise if $v^2=(0,0,0)$.
So we may assume that both vectors $v^1,v^2$ are non-zero.

By the linear dependency, $v^2$ is a multiple of $v^1$.
As both $v^1,v^2$ have integer entries, the multiple must be given by a rational number.
So there are $k_1,k_2\in\mathbb Z\setminus\{0\}$ such that $k_1v^1=k_2v^2$.
As $[v^1]\neq\emptyset$, there is $s\in\mathcal K_3$ such that $v^1$ is a cycle for $s$.
By \cref{l:multiple}, $k_1v^1=k_2v^2$ is also a cycle for $s$. Another application of \cref{l:multiple} gives us that $v^2$ is a cycle for $s$.
So $s\in[v^1]\cap[v^2]$.

Now suppose that $[v^1]\cap[v^2]\neq\emptyset$. Then $[v^1]\neq\emptyset$ and $[v^2]\neq\emptyset$.
If one (or both) of $v^1,v^2$ is the zero vector, then $v^1,v^2$ are linearly dependent and we are done.
So we may assume that both $v^1,v^2$ are non-zero.
Denote $v^1=(l,m,r)$ and $v^2=(L,M,R)$.
Let $s\in\mathcal K_3$ be such that both $v^1,v^2$ are cycles for $s$.
Then, by \cref{l:multiple}, $Lv^1$ and $lv^2$ are also cycles for $s$.
By \cref{l:plusMinus}, the vector
\[Lv^1-lv^2=(0,Lm-lM,Lr-lR)\]
is also a cycle for $s$.
By \cref{l:charNonempty}, we must have $l m r \neq 0$, $L M R \neq 0$ and
\[Lm-lM=Lr-lR=0.\]
This means that the determinants of the matrices
\begin{equation*}
\begin{bmatrix}
l & m \\ L & M
\end{bmatrix}
\text{ and }
\begin{bmatrix}
l & r \\ L & R
\end{bmatrix}
\end{equation*}
are zero, and, as $l L\neq 0$, we immediately obtain that $v^1,v^2$ are linearly dependent.
\end{proof}

\begin{lemma}\label{l:kernel}
Let $s=(C,(c_1,c_2,c_3))\in\mathcal K_3$ be given.
Then there is $v^0\in\mathbb Z^3$ such that
\begin{equation}\label{e:kv_0}
\left\{v\in\mathbb Z^3:s\in [v]\right\}=\left\{kv^0:k\in\mathbb Z\right\}.
\end{equation}
Moreover, either $\gcd(v^0)=1$, or $v^0=(0,0,0)$.
\end{lemma}

\begin{proof}
Note that $(0,0,0)$ is trivially a cycle for $s$.
If $\{v\in\mathbb Z^3:s\in [v]\}=\{(0,0,0)\}$,
then~\eqref{e:kv_0} holds for $v^0=(0,0,0)$.
So we may assume that $\{v\in\mathbb Z^3:s\in [v]\}\neq\{(0,0,0)\}$.
Then there is $v\neq(0,0,0)$ which is a cycle for $s$.
By \cref{l:charNonempty}, all three entries of $v$ are non-zero.
Let 
\begin{equation*}
v^0=\frac v{\gcd(v)},
\end{equation*}
so that $\gcd(v^0)=1$.
Then $kv^0$ is a cycle for $s$ for every $k\in\mathbb Z$ by \cref{l:multiple}.
It remains to show that there are no other cycles for $s$.

So let $v'\in\mathbb Z^3\setminus\{(0,0,0)\}$ be an arbitrary cycle for $s$.
By \cref{l:dependency}, $v^0$ and $v'$ are linearly dependent.
So there are $k_1,k_2\in\mathbb Z\setminus\{0\}$ such that $k_1v^0=k_2v'$.
As $\gcd(v^0)=1$, we have $\gcd(k_2v')=\gcd(k_1v^0)=k_1$.
But all entries of $k_2v'$ are divisible by $k_2$, so $k_2$ divides $k_1$.
Thus $\frac{k_1}{k_2}\in\mathbb Z$, and then
\begin{equation*}
v'=\frac{k_1}{k_2}v^0\in\left\{kv^0:k\in\mathbb Z\right\},
\end{equation*}
which ends the proof.
\end{proof}

\begin{lemma}\label{l:only1choice}
Let $K=\{(x,y)\in\mathbb R^2:x^2+y^2=1\}$ be the unit circle in the plane and let $-1<a_1<a_2<1$ be fixed.
For every $a_3\in(a_2,1)$, let 
$s_{a_3}\in\mathcal K_3$ be given by
\begin{equation*}
s_{a_3}=(K,((a_1,0),(a_2,0),(a_3,0))).	
\end{equation*}
Then, for every $v\in\mathbb Z^3\setminus\{(0,0,0)\}$, there exists at most one $a_3\in(a_2,1)$ such that $v$ is a cycle for $s_{a_3}$.
\end{lemma}

\begin{proof}
As in the proof of \Cref{l:charNonempty}, for every $a\in(-1,1)$, we will use the symbol $R_a$ for the reversion $P_{(a,0)}\colon K\to K$ through the point $(a,0)$.
For notational purposes, we define $a_1':=a_1$ and $a_2':=a_2$.

Suppose for a contradiction that $a_3,a_3'\in(a_2,1)$, $a_3<a_3'$, are such that some $v=(v_1,v_2,v_3)\neq(0,0,0)$ is a cycle for both $s_{a_3}$ and $s_{a_3'}$.
Then we can find $g=(i_1,\ldots,i_n)\in G_3$ such that $N(g)=v$. Note that $n>0$ as $v\neq(0,0,0)$, and that $v_1v_2v_3\neq 0$ by \cref{l:charNonempty}. We may assume that, for each $i\in\{1,2,3\}$, $g$ contains $i$ only at even positions, or only at odd positions (if it is not the case, just replace $g$ by $\overline g$).

Let $m\in\{1,\ldots,n\}$ be the first index such that $i_m=3$.
Such an index exists as $v_3\neq 0$.
Then
\begin{equation*}
\left(R_{a_{i_{m-1}}}\circ\ldots\circ R_{a_{i_1}}\right)(0,1)=\left(R_{a'_{i_{m-1}}}\circ\ldots\circ R_{a'_{i_1}}\right)(0,1).	
\end{equation*}
Consequently, by our assumption $a_3'>a_3$, the first coordinate of
\begin{equation*}
\left(R_{a'_{i_m}}\circ\ldots\circ R_{a'_{i_1}}\right)(0,1)=R_{a'_3}\left(\left(R_{a'_{i_{m-1}}}\circ\ldots\circ R_{a'_{i_1}}\right)(0,1)\right)	
\end{equation*}
is strictly bigger that the first coordinate of
\begin{equation*}
\left(R_{a_{i_m}}\circ\ldots\circ R_{a_{i_1}}\right)(0,1)=R_{a_3}\left(\left(R_{a_{i_{m-1}}}\circ\ldots\circ R_{a_{i_1}}\right)(0,1)\right).	
\end{equation*}
Now, by induction on $k$, the reader may easily verify that, for every $k\in\{0,\ldots,n-m\}$ (the case $k=0$ was verified just now), it holds:
\begin{itemize}
\item if $k$ is even then the first coordinate of $\left(R_{a'_{i_{m+k}}}\circ\ldots\circ R_{a'_{i_1}}\right)(0,1)$ is strictly bigger that the first coordinate of $\left(R_{a_{i_{m+k}}}\circ\ldots\circ R_{a_{i_1}}\right)(0,1)$,
\item if $k$ is odd then the first coordinate of $\left(R_{a'_{i_{m+k}}}\circ\ldots\circ R_{a'_{i_1}}\right)(0,1)$ is strictly smaller that the first coordinate of $\left(R_{a_{i_{m+k}}}\circ\ldots\circ R_{a_{i_1}}\right)(0,1)$;
\end{itemize}
the key ingredients of the induction step are the following facts:
\begin{itemize}
\item if $i_{m+k}=3$ then $k$ is even (by our assumption on $g$),
\item any reversion restricted to the upper half-circle is strictly decreasing in terms of the first coordinate, the same is true for the restriction to the lower half-circle,
\item for any fixed $x\in K\setminus\{(-1,0),(1,0)\}$, the map $"(-1,1)\ni a\mapsto R_a(x)\in K"$ is strictly increasing in terms of the first coordinate.
\end{itemize}

Finally, the case $k=n-m$ implies that
\begin{equation*}
\left(R_{a'_{i_n}}\circ\ldots\circ R_{a'_{i_1}}\right)(0,1)\neq \left(R_{a_{i_n}}\circ\ldots\circ R_{a_{i_1}}\right)(0,1).	
\end{equation*}
So either
\begin{equation*}
\left(R_{a'_{i_n}}\circ\ldots\circ R_{a'_{i_1}}\right)(0,1)\neq(0,1),
\end{equation*}
or
\begin{equation*}
\left(R_{a_{i_n}}\circ\ldots\circ R_{a_{i_1}}\right)(0,1)\neq(0,1).
\end{equation*}
But by \cref{l:cyclePoint}, this is a contradiction with our assumption that $v$ is a cycle for both $s_{a_3}$ and $s_{a_3'}$.
\end{proof}

Finally, we are ready for the proof of our main result.

\begin{proof}[Proof of \cref{t:isoclasses}]
Let $C$ be a circle in the plane and $c_1,c_2,c_3$ be pairwise distinct collinear points in the interior of $C$.
We start by showing that the set $C\cup\{c_1,c_2,c_3\}$ belongs to (at least) one of the sets \eqref{eq:firstline} and
\eqref{eq:secondline}.
Without loss of generality,
we may assume that $c_2$ lies on the line segment connecting $c_1$ and $c_3$.
Then there are exactly two distinct $s,s'\in\mathcal K_3$ such that $\set(s)=\set(s')=C\cup\{c_1,c_2,c_3\}$, namely $(C,(c_1,c_2,c_3))$ and $(C,(c_3,c_2,c_1))$.
Since we have just fixed the circle $C$ and the points $c_1,c_2,c_3$,
we have now the group action of $G_3$ on $C$ and the associated notions 
of orbits and centers (\cref{subsec:group,subsec:reversionETC})
as well as the meaning of cycles and $[v]$, $[[v]]$ for $v\in \Z^3$
(\cref{subsec:l3}).
It follows by the definitions that
\begin{equation}\label{e:twoCases}
\left\{v\in\mathbb Z^3:C\cup\{c_1,c_2,c_3\}\in[[v]]\right\}=
\left\{v\in\mathbb Z^3:s\in[v]\right\}\cup
\left\{v\in\mathbb Z^3:s'\in[v]\right\};
\end{equation}
there is an obvious relation between the two sets on the right hand side: with $\pi_{1,3}(v)$ defined as $(v_3,v_2,v_1)\in\mathbb Z^3$ for every $v=(v_1,v_2,v_3)\in\mathbb Z^3$, we have
\begin{equation}\label{e:aboutTwoCases}
\left\{v\in\mathbb Z^3:s'\in[v]\right\}=
\pi_{1,3} \left(\left\{v\in\mathbb Z^3:s\in[v]\right\}\right).
\end{equation}
We also note that
$\gcd(\pi_{1,3}(v))=\gcd(v)$ and
\begin{equation}\label{e:aboutPi}
\pi_{1,3}(k v)=k \pi_{1,3}(v)
\qquad\text{for every $v\in \mathbb Z^3$ and $k\in \mathbb Z$.}
\end{equation}

By \cref{l:kernel}, there is $v^0=(v^0_1,v^0_2,v^0_3)\in\mathbb Z^3$, with either $\gcd(v^0)=1$ or $v^0=(0,0,0)$, such that
\begin{equation*}
\left\{v\in\mathbb Z^3:s\in[v]\right\}=\left\{kv^0:k\in\mathbb Z\right\}.	
\end{equation*}
Without loss of generality, we may assume that $v^0_2\ge 0$ (otherwise, we can replace $v^0$ by $-v^0$).
Using~\eqref{e:twoCases}, \eqref{e:aboutTwoCases} and~\eqref{e:aboutPi}, it follows that
\begin{equation}\label{e:svsS}
\left\{v\in\mathbb Z^3:C\cup\{c_1,c_2,c_3\right\}\in[[v]]\}=\left\{kv^0:k\in\mathbb Z\right\}\cup\left\{k\pi_{1,3}(v^0):k\in\mathbb Z\right\}.
\end{equation}
We immediately see that if $v^0=(0,0,0)$, then
\begin{equation*}
\left\{v\in\mathbb Z^3:C\cup\{c_1,c_2,c_3\}\in[[v]]\right\}=\left\{(0,0,0))\right\},
\end{equation*}
and so $C\cup\{c_1,c_2,c_3\}\in\mathbb O$.
So suppose that $v^0\neq(0,0,0)$. Then, by \cref{l:charNonempty}, it holds $v^0_1v^0_2v^0_3\neq 0$, $|v^0_2|=|v^0_1|+|v^0_3|$ and $v^0$ is balanced (then also $\pi_{1,3}(v^0)$ is balanced). Also, at least one of the vectors $v^0,\pi_{1,3}(v^0)$, satisfies that its first coordinate is less than or equal to its third coordinate. Let $v'=(v'_1,v'_2,v'_3)$ be such a vector.
Then $v'\in\mathbb Z^3$ is a balanced vector such that $v'_2\ge 1$, $v'_1\le v'_3\le -1$, $\gcd(v')=1$ and $C\cup\{c_1,c_2,c_3\}\in[[v']]$.

We have shown that every set $C\cup\{c_1,c_2,c_3\}$, as in the statement of the theorem, belongs to (at least) one of the sets \eqref{eq:firstline}, \eqref{eq:secondline}.
Next, we show that each of the sets~\eqref{eq:firstline}, \eqref{eq:secondline} is non-empty.
For~\eqref{eq:firstline}, this follows from \cref{l:charNonempty}.
To prove that $\mathbb O\neq\emptyset$, let $K=\{(x,y)\in\mathbb R^2:x^2+y^2=1\}$ be the unit circle in the plane.

For every $v\in\mathbb Z^3\setminus\{(0,0,0)\}$, we know by \cref{l:only1choice} that there is at most one $a_v\in(0,1)$ such that $v$ is a cycle for $(K,((-\frac12,0),(0,0),(a_v,0)))$.
If there is no such $a_v$ then we define $a_v:=\frac 12$ (or we can choose any other value from the interval $(0,1)$).
Now we fix some
\begin{equation*}
a\in(0,1)\setminus\left\{a_v:v\in\mathbb Z^3\setminus\{(0,0,0)\}
\right\}.	
\end{equation*}
Then no $v\in\mathbb Z^3\setminus\{(0,0,0)\}$ is a cycle for $s:=(K,((-\frac12,0),(0,0),(a,0)))$.
Hence $\set(s)\in \mathbb O$ by the definition of $\mathbb O$, \eqref{eq:secondline},
so $\mathbb O\neq\emptyset$.

\smallbreak
Now we show that the sets \eqref{eq:firstline}, \eqref{eq:secondline} are pairwise disjoint (and, consequently, each of them is listed just once as they are non-empty).

If $S\in\mathbb O$, then there is $s=(C,(c_1,c_2,c_3))\in\mathcal K_3$ such that $(0,0,0)$ is the only cycle for $s$.
Then $s':=(C,(c_3,c_2,c_1))$ is the only element of $\mathcal K_3$, distinct from $s$, with $\set(s')=S$.
Obviously, $(0,0,0)$ is the only cycle also for $s'$. 
Now if $S\in[[v]]$ (where $v$ is as in \eqref{eq:firstline})
then there is $s\in\mathcal K_3$ with $\set(s)=S$ such that $s\in[v]$.
So $\mathbb O$ is disjoint from each of the sets $[[v]]$ in \eqref{eq:firstline}.

Now suppose that $[[v^1]],[[v^2]]$ are as in \eqref{eq:firstline}
(then, in particular, $\gcd(v^1)=1=\gcd(v^2)$)
and that some set $C\cup\{c_1,c_2,c_3\}$, as in the statement of the theorem, belongs to both $[[v^1]],[[v^2]]$.
By what we already proved, we know that there is some $v^0=(v^0_1,v^0_2,v^0_3)\in\mathbb Z^3$
satisfying~\eqref{e:svsS}.
By~\eqref{e:svsS}, each of the vectors $v^1,v^2$ is either an integer multiple of $v^0$, or an integer multiple of $\pi_{1,3}(v^0)$.
As $\gcd(v^1)=1$, %
this implies that $v^1 = \pm v^0$ or $v^1 = \pm \pi_{1,3}(v^0)$ and likewise for $v^2$.
Hence $v^1 = \pm v^2$ or $v^1 = \pm \pi_{1,3}(v^2)$. 
From the inequalities of~\eqref{eq:firstline} we then get $v^1=v^2$.
This completes the proof of the disjointness.

\smallbreak
It remains to show that each set of the form~\eqref{eq:firstline} or~\eqref{eq:secondline} is a betweenness isomorphism class of the collection of all sets $C\cup\{c_1,c_2,c_3\}$ as in the statement of the theorem.

\smallbreak
Let $S,R$ be sets of the form $S=C\cup\{c_1,c_2,c_3\}$ and $R=D\cup\{d_1,d_2,d_3\}$, where $C,D$ are circles in the plane, $c_1,c_2,c_3$ are pairwise distinct collinear points from the interior of $C$, and $d_1,d_2,d_3$ are pairwise distinct collinear points from the interior of $D$.
To complete the proof, we must show that the sets $S,R$ are betweenness isomorphic if and only if they belong to the same set from the partition given by~\eqref{eq:firstline} and~\eqref{eq:secondline}.

Suppose first that $S,R$ are betweenness isomorphic.
We will show that if $S\in[[v]]$ for some $v$ (as in~\eqref{eq:firstline}) then $R\in[[v]]$.
Similarly, one could prove that if $R\in[[v]]$ then $S\in[[v]]$.
It will follow that $R,S$ belong to the same sets of the form~\eqref{eq:firstline}.
And as the sets~\eqref{eq:firstline}, \eqref{eq:secondline} form a partition, $S\in\mathbb O$ if and only if $R\in\mathbb O$.
So suppose that $S\in[[v]]$ for some $v$ as in~\eqref{eq:firstline}.
Then there is $s\in\mathcal K_3$ with $\set(s)=S$ such that $s\in[v]$.
By rearranging the order in which the points $c_1,c_2,c_3$ are indexed, we may assume that $s=(C,(c_1,c_2,c_3))$.
Let $f$ be a betweenness isomorphism from $S$ to $R$.
By rearranging the order in which the points $d_1,d_2,d_3$ are indexed, we may assume that the permutation $\sigma\colon\{1,2,3\}\to\{1,2,3\}$ which corresponds to $f$ is the identity permutation (that is, $f$ maps $c_i$ to $d_i$, $i\in\{1,2,3\}$).
Then, by \cref{c:characterization} and \cref{r:Afterchara},  there is a bijection $h\colon\mathcal O_C\to\mathcal O_D$ such that, for every $\mathbf O\in\mathcal O_C$, there are $x_{\mathbf O}\in\mathbf O$ and $y_{\mathbf O}\in h(\mathbf O)$ with
\begin{equation}\label{e:equationForStabilizers}
\stab(y_{\mathbf O})=\stab(x_{\mathbf O}).
\end{equation}
As $s\in[v]$, there are $g\in G_3$ and $c\in C\setminus\chS(\{c_1,c_2,c_3\})$ such that $N(g)=v$ and $g\in\stab(c)$.
Fix some $\orb\in\mathcal O_C$.
By \cref{l:cyclePoint}, it holds
\begin{equation*}
g\in\stab(x_{\orb})\stackrel{\eqref{e:equationForStabilizers}}{=}\stab(y_{\orb}).
\end{equation*} 
So, $g$ and $y_\orb$ witness that
\begin{equation*}
(D,(d_1,d_2,d_3))\in[v].	
\end{equation*}
In particular, $R\in[[v]]$, as we wanted.

\smallbreak
Now suppose that $S,R$ belong to the same set from the partition given by~\eqref{eq:firstline} and~\eqref{eq:secondline}.
We first deal with the case that $S,R\in[[v]]$ for some $v$ as in~\eqref{eq:firstline}.
Then there are $s,r\in\mathcal K_3$ such that $\set(s)=S$, $\set(r)=R$ and $s,r\in[v]$.
By rearranging the order in which the points $c_1,c_2,c_3$, resp. $d_1,d_2,d_3$, are indexed, we may assume that
\begin{equation*}
s=(C,(c_1,c_2,c_3))\quad\text{and}\quad r=(D,(d_1,d_2,d_3)).	
\end{equation*}
By \cref{l:kernel}, there are $v^s,v^r\in\mathbb Z^3$ such that
\begin{equation*}
\left\{u\in\mathbb Z^3:s\in[u]\right\}=\left\{kv^s:k\in\mathbb Z\right\}
\end{equation*}
and
\begin{equation*}
\left\{u\in\mathbb Z^3:r\in[u]\right\}=\left\{kv^r:k\in\mathbb Z\right\}.
\end{equation*}
So there are $k_s,k_r\in\mathbb Z$ such that
\begin{equation}\label{e:k_sv^s}
k_sv^s=v=k_rv^r.
\end{equation}
In particular, the vectors $v^s,v^r$ are linearly dependent.
\cref{l:kernel} also states that either $\gcd(v^s)=1$, or $v^s=(0,0,0)$. The latter case cannot happen by~\eqref{e:k_sv^s} as $v\neq(0,0,0)$, so $\gcd(v^s)=1$. Similarly, $\gcd(v^r)=1$.
Now, \eqref{e:k_sv^s} easily implies that $k_s=\pm k_r$ and, consequently, $v^s=\pm v^r$.
So it holds
\begin{equation}\label{e:rovnostCyklu}
\left\{u\in\mathbb Z^3:s\in[u]\right\}=\left\{u\in\mathbb Z^3:r\in[u]\right\}.
\end{equation}

We claim that, whenever $c\in C\setminus\chS(\{c_1,c_2,c_3\})$ and $d\in D\setminus\chR(\{d_1,d_2,d_3\})$, we have
\[\stab(c)=\stab(d).\]
By symmetry, it is enough to show only one inclusion.
So fix $c,d$ as above and suppose that $g\in G_3$ belongs to $\stab(c)$.
Then $N(g)$ is a cycle for $s$.
By \cref{e:rovnostCyklu}, $N(g)$ is also a cycle for $r$.
Now \cref{l:cyclePoint} gives us that $g\in\stab(d)$.
So our claim is verified.

Observe that $c_2$ lies on the line segment connecting $c_1$ and $c_3$ (as $s=(C,(c_1,c_2,c_3))\in\mathcal K_3$), similarly for $d_2$.
Hence, there is a betweenness isomorphism $\tilde f\colon\chS(\{c_1,c_2,c_3\})\to \chR(\{d_1,d_2,d_3\})$ which maps $c_i$ to $d_i$ for every $i\in\{1,2,3\}$.
Then the permutation $\sigma\colon\{1,2,3\}\to\{1,2,3\}$ which corresponds to $\tilde f$ is the identity permutation.

Note that each orbit from $\mathcal O_C$, as well as each orbit from $\mathcal O_D$, is at most countable. So there are continuum many orbits from $\mathcal O_C$, as well as from $\mathcal O_D$.
So we can fix some bijection $h\colon\mathcal O_C\to\mathcal O_D$.
By the preceding claim, for every $\mathbf O\in\mathcal O_C$, there are $x_{\mathbf O}\in\mathbf O$ and $y_{\mathbf O}\in h(\mathbf O)$ such that
\begin{equation*}
\stab(y_{\mathbf O})=\stab(x_{\mathbf O}).
\end{equation*}
So we can apply \cref{c:characterization} to conclude that $S,R$ are betweenness isomorphic.

Finally, suppose that $S,R\in\mathbb O$.
Then there are $s,r\in\mathcal K_3$ such that 
$\set (s)=S$, $\set (r)=R$ and
$(0,0,0)$ is the only cycle for $s$, as well as the only cycle for $r$.
By the definition of a cycle and by \cref{l:cyclePoint}, for every $c\in C\setminus\chu_S(\{c_1,c_2,c_3\})$ and every $d\in D\setminus\chu_R(\{d_1,d_2,d_3\})$, we have
\begin{equation*}
\stab(c)=\left\{g\in G_3:N(g)=(0,0,0)\right\}=\stab(d).
\end{equation*}
Hence, we can apply \cref{c:characterization} in the same way as above to conclude that $S$ and $R$ are betweenness isomorphic.
\end{proof}

\bibliographystyle{alpha}
\bibliography{DKM-bibliography}

\begin{thebibliography}{ZPFDB19}

\bibitem[ABM21]{AndersonBankstonMcCluskey2021}
Daron Anderson, Paul Bankston, and Aisling McCluskey.
\newblock Convexity in topological betweenness structures.
\newblock {\em Topology Appl.}, 304:Paper No. 107783, 20, 2021.

\bibitem[AN98]{AdelekeNeumann1998}
S.~A. Adeleke and Peter~M. Neumann.
\newblock Relations related to betweenness: their structure and automorphisms.
\newblock {\em Mem. Amer. Math. Soc.}, 131(623):viii+125, 1998.

\bibitem[Ban13]{Bankston2013}
Paul Bankston.
\newblock Road systems and betweenness.
\newblock {\em Bull. Math. Sci.}, 3(3):389--408, 2013.

\bibitem[Ban15]{Bankston2015}
Paul Bankston.
\newblock The antisymmetry betweenness axiom and {H}ausdorff continua.
\newblock {\em Topology Proc.}, 45:189--215, 2015.

\bibitem[BM23]{BankstonMcCluskey2023}
Paul Bankston and Aisling McCluskey.
\newblock On betweenness and equidistance in metric spaces.
\newblock {\em J. Convex Anal.}, 30(1):371--400, 2023.

\bibitem[BMS17]{BrunoMcCluskeySzeptycki2017}
J.~Bruno, A.~McCluskey, and P.~Szeptycki.
\newblock Betweenness relations in a categorical setting.
\newblock {\em Results Math.}, 72(1-2):649--664, 2017.

\bibitem[Bog97]{Bogomolny1997}
A.~Bogomolny.
\newblock {Two Butterflies Theorem. Interactive Mathematics Miscellany and
  Puzzles}.
\newblock
  \url{https://www.cut-the-knot.org/Curriculum/Geometry/TwoButterflies.shtml},
  1997.
\newblock Accessed 3 October 2023.

\bibitem[Chv09]{Chvatal2009}
Va\v{s}ek Chv\'{a}tal.
\newblock Antimatroids, betweenness, convexity.
\newblock In {\em Research trends in combinatorial optimization}, pages 57--64.
  Springer, Berlin, 2009.

\bibitem[CKL13]{ChajdaKolarikLanger2013}
Ivan Chajda, Miroslav Kola\v{r}\'{\i}k, and Helmut L\"{a}nger.
\newblock Algebras assigned to ternary relations.
\newblock {\em Miskolc Math. Notes}, 14(3):827--844, 2013.

\bibitem[CNSS19]{ChangatNarasimha-ShenoiSeethakuttyamma2019}
Manoj Changat, Prasanth~G. Narasimha-Shenoi, and Geetha Seethakuttyamma.
\newblock Betweenness in graphs: a short survey on shortest and induced path
  betweenness.
\newblock {\em AKCE Int. J. Graphs Comb.}, 16(1):96--109, 2019.

\bibitem[Cou20]{Courcelle2020}
Bruno Courcelle.
\newblock Betweenness in order-theoretic trees.
\newblock In {\em Fields of logic and computation. {III}}, volume 12180 of {\em
  Lecture Notes in Comput. Sci.}, pages 79--94. Springer, Cham, [2020]
  \copyright 2020.

\bibitem[Cou21]{Courcelle2021}
Bruno Courcelle.
\newblock Axiomatizations of betweenness in order-theoretic trees.
\newblock {\em Log. Methods Comput. Sci.}, 17(1):Paper No. 11, 42, 2021.

\bibitem[CW12]{ChvatalWu2012}
Va\v{s}ek Chv\'{a}tal and Baoyindureng Wu.
\newblock On {R}eichenbach's causal betweenness.
\newblock {\em Erkenntnis}, 76(1):41--48, 2012.

\bibitem[DW81]{DiminnieWhite1981}
Charles~R. Diminnie and Albert~G. White.
\newblock Remarks on strict convexity and betweenness postulates.
\newblock {\em Demonstratio Math.}, 14(1):209--220, 1981.

\bibitem[DW04]{DuvelmeyerWenzel2004}
Nico D\"{u}velmeyer and Walter Wenzel.
\newblock A characterization of ordered sets and lattices via betweenness
  relations.
\newblock {\em Results Math.}, 46(3-4):237--250, 2004.

\bibitem[Fis71]{Fishburn1971}
Peter~C. Fishburn.
\newblock Betweenness, orders and interval graphs.
\newblock {\em J. Pure Appl. Algebra}, 1(2):159--178, 1971.

\bibitem[Hed81]{Hedlikova1981}
J.~Hedl\'{\i}kov\'{a}.
\newblock Betweenness isomorphisms of modular lattices.
\newblock {\em Arch. Math. (Basel)}, 37(2):154--162, 1981.

\bibitem[Hed83]{Hedlikova1983}
Jarmila Hedl\'{\i}kov\'{a}.
\newblock Ternary spaces, media, and {C}hebyshev sets.
\newblock {\em Czechoslovak Math. J.}, 33(108)(3):373--389, 1983.

\bibitem[HK17]{HuntingtonKline1917}
Edward~V. Huntington and J.~Robert Kline.
\newblock Sets of independent postulates for betweenness.
\newblock {\em Trans. Amer. Math. Soc.}, 18(3):301--325, 1917.

\bibitem[HM08]{HouMcColm2008}
Xiang-Dong Hou and Gregory McColm.
\newblock When is betweeness preserved?
\newblock {\em Rocky Mountain J. Math.}, 38(1):123--137, 2008.

\bibitem[Hun24]{Huntington1924}
Edward~V. Huntington.
\newblock A new set of postulates for betweenness, with proof of complete
  independence.
\newblock {\em Trans. Amer. Math. Soc.}, 26(2):257--282, 1924.

\bibitem[JW23]{JostWenzel2023}
J\"{u}rgen Jost and Walter Wenzel.
\newblock Geometric algebra for sets with betweenness relations.
\newblock {\em Beitr. Algebra Geom.}, 64(3):555--579, 2023.

\bibitem[KMZ22]{KubisMorawiecZurcher2022}
Wies{\l}aw Kubi\'{s}, Janusz Morawiec, and Thomas Z\"{u}rcher.
\newblock {Monotone mappings and lines}.
\newblock \url{https://doi.org/10.48550/arXiv.2004.14301}, 2022.
\newblock Manuscript.

\bibitem[Koc13]{Kocik2013}
Jerzy Kocik.
\newblock A porism concerning cyclic quadrilaterals.
\newblock {\em Geometry}, 2013:5, 2013.
\newblock Id/No 483727.

\bibitem[Kub02]{Kubis2002}
Wies{\l}aw Kubi\'{s}.
\newblock Separation properties of convexity spaces.
\newblock {\em J. Geom.}, 74(1-2):110--119, 2002.

\bibitem[Lih00]{Lihova2000}
Judita Lihov\'{a}.
\newblock Strict order-betweennesses.
\newblock {\em Acta Univ. M. Belii Ser. Math.}, 8:27--33, 2000.

\bibitem[Maz14]{Mazurkiewicz1914}
S.~Mazurkiewicz.
\newblock O pewney mnogo{\'{s}}ci p{\l}askiej, kt{\'{o}}ra ma z ka{\.z}d{\c a}
  prost{\c a} dwa i tylko dwa punkty wspólne.
\newblock {\em Comptes Rendus, Soci\'{e}t\'{e} des Sciences (et des Lettres) de
  Varsovie, Serie A}, 7:322--383, 1914.

\bibitem[MM02]{MorganaMulder2002}
Maria~Aurora Morgana and Henry~Martyn Mulder.
\newblock The induced path convexity, betweenness, and svelte graphs.
\newblock {\em Discrete Math.}, 254(1-3):349--370, 2002.

\bibitem[Pam11]{Pambuccian2011}
Victor Pambuccian.
\newblock The axiomatics of ordered geometry {I}. {O}rdered incidence spaces.
\newblock {\em Expo. Math.}, 29(1):24--66, 2011.

\bibitem[Pas82]{Pasch1882}
M.~Pasch.
\newblock Vorlesungen {\"u}ber neuere {Geometrie}.
\newblock Leipzig. {Teubner} (1882)., 1882.

\bibitem[Sha23]{Shakir2023}
Qays~R. Shakir.
\newblock On setwise betweenness.
\newblock {\em Appl. Gen. Topol.}, 24(1):115--123, 2023.

\bibitem[Shi22]{Shi2022}
Yi~Shi.
\newblock Betweenness relations and gated sets in fuzzy metric spaces.
\newblock {\em Fuzzy Sets and Systems}, 437:1--19, 2022.

\bibitem[Sho52]{Sholander1952}
Marlow Sholander.
\newblock Trees, lattices, order, and betweenness.
\newblock {\em Proc. Amer. Math. Soc.}, 3:369--381, 1952.

\bibitem[Sim09]{Simovici2009}
Dan~A. Simovici.
\newblock Betweenness, metrics and entropies in lattices.
\newblock {\em J. Mult.-Valued Logic Soft Comput.}, 15(4):409--420, 2009.

\bibitem[Smi43]{Smiley1943}
M.~F. Smiley.
\newblock A comparison of algebraic, metric, and lattice betweenness.
\newblock {\em Bull. Amer. Math. Soc.}, 49:246--252, 1943.

\bibitem[Sol84]{Soltan1984}
V.~P. Soltan.
\newblock {\em \selectlanguage{russian}Введение в
  аксиоматическую теорию выпуклости. (Russian)
  [Introduction to the axiomatic theory of convexity]\selectlanguage{english}}.
\newblock ``Shtiintsa'', Kishinev, 1984.
\newblock With English and French summaries.

\bibitem[ST43]{SmileyTransue1943}
M.~F. Smiley and W.~R. Transue.
\newblock Applications of transitivities of betweenness in lattice theory.
\newblock {\em Bull. Amer. Math. Soc.}, 49:280--287, 1943.

\bibitem[Tor71]{Toranzos1971}
F.~A. Toranzos.
\newblock Metric betweenness in normed linear spaces.
\newblock {\em Colloq. Math.}, 23:99--102, 1971.

\bibitem[vdV93]{Vel1993}
M.~L.~J. van~de Vel.
\newblock {\em Theory of convex structures}, volume~50 of {\em North-Holland
  Mathematical Library}.
\newblock North-Holland Publishing Co., Amsterdam, 1993.

\bibitem[ZPFDB19]{ZhangPerez-FernandezBeats2019}
Hua-Peng Zhang, Ra\'{u}l P\'{e}rez-Fern\'{a}ndez, and Bernard De~Baets.
\newblock Topologies induced by the representation of a betweenness relation as
  a family of order relations.
\newblock {\em Topology Appl.}, 258:100--114, 2019.

\bibitem[ZZZ23]{ZhaoZhaoAhang2023}
Hu~Zhao, Yu-Jie Zhao, and Shao-Yu Zhang.
\newblock On {$(L, N)$}-fuzzy betweenness relations.
\newblock {\em Filomat}, 37(11):3559--3573, 2023.

\end{thebibliography}

\end{document}